\documentclass[11 pt, reqno, twoside]{amsart}

\usepackage[latin1]{inputenc}
\usepackage{lmodern}
\usepackage[T1]{fontenc}
\usepackage[english]{babel}
\usepackage{ngerman}
\usepackage{amsmath}
\usepackage{amssymb}
\usepackage{amsfonts}
\usepackage{amstext}
\usepackage{amsthm}
\usepackage{hyperref}
\usepackage{xcolor}
\usepackage{graphicx}
\usepackage{subfigure}

\usepackage[a4paper, hdivide={3.0 cm,,3.0cm},vdivide={3.5cm,,4.2cm}]{geometry}

\allowdisplaybreaks

\begin{document}

\theoremstyle{plain}
	\newtheorem{Pp}{Proposition}[section]
	\newtheorem{Thm}[Pp]{Theorem}
	\newtheorem{Lm}[Pp]{Lemma}
	\newtheorem{Cor}[Pp]{Corollary}
\theoremstyle{definition}
	\newtheorem{Df}[Pp]{Definition}
	\newtheorem{Cond}[Pp]{Condition}
	\newtheorem{Ass}[Pp]{Assumption}
	\newtheorem{Rm}[Pp]{Remark}
	\newtheorem{Emp}[Pp]{}

\title[Geometric Langevin equations and applications]{Geometric Langevin equations on submanifolds\\and applications to the stochastic melt-spinning process of nonwovens and biology}

\author{Martin Grothaus}
\address{
Martin Grothaus, Mathematics Department, University of Kaiserslautern, \newline 
P.O.Box 3049, 67653 Kaiserslautern, Germany.\newline 
{\rm \texttt{Email:~grothaus@mathematik.uni-kl.de}},\newline
Functional Analysis and Stochastic Analysis Group, \newline
{\rm \texttt{URL:~http://www.mathematik.uni-kl.de/$\sim$wwwfktn/ }}} 

\author{Patrik Stilgenbauer}
\address{
Patrik Stilgenbauer, Mathematics Department, University of Kaiserslautern, \newline
P.O.Box 3049, 67653 Kaiserslautern, Germany. \newline
{\rm \texttt{Email:~stilgenb@mathematik.uni-kl.de}}, \newline
Functional Analysis and Stochastic Analysis Group, \newline
{\rm \texttt{URL:~http://www.mathematik.uni-kl.de/$\sim$wwwfktn/ }}}

\date{\today}



\begin{abstract}
In this article we develop geometric versions of the classical Langevin equation on regular submanifolds in euclidean space in an easy, natural way and combine them with a bunch of applications. The equations are formulated as Stratonovich stochastic differential equations on manifolds. The first version of the geometric Langevin equation has already been detected before, see \cite{LRS12} for a different derivation. We propose an additional extension of the models, the geometric Langevin equations with velocity of constant absolute value. The latters are seemingly new and provide a galaxy of new, beautiful and powerful mathematical models. Up to the authors best knowledge there are not many mathematical papers available dealing with geometric Langevin processes. We connect the first version of the geometric Langevin equation via proving that its generator coincides with the generalized Langevin operator proposed in \cite{Sol95}, \cite{Jor78} or \cite{Kol00}. All our studies are strongly motivated by industrial applications in modeling the fiber lay-down dynamics in the production process of nonwovens. We light up the geometry occuring in these models and show up the connection with the spherical velocity version of the geometric Langevin process. Moreover, as a main point, we construct new smooth industrial relevant three-dimensional fiber lay-down models involving the spherical Langevin process. Finally, relations to a class of self-propelled interacting particle systems with roosting force are presented and further applications of the geometric Langevin equations are given.
\end{abstract}

\maketitle

\section{Introduction} \label{Introduction}
This article is about geometric Langevin equations on regular submanifolds of euclidean space and its applications. As starting point, let us consider the classical Langevin equation
\begin{align} \label{Langevin_in _Rd}
&\mathrm{d}\xi_t = \omega_t \, \mathrm{dt} \\
&\mathrm{d}\omega_t = -\lambda \,\omega_t \, \mathrm{dt} - \nabla \Phi(\xi_t)\, \mathrm{dt} + \sigma\, \mathrm{d}W_t \nonumber
\end{align}
in $\mathbb{R}^{2d}$, $d \in \mathbb{N}$. For simplicity $\Phi \in C^\infty(\mathbb{R}^d)$ and $\lambda, \sigma \in (0,\infty)$. Besides its crucial applications in statistical physics we regard this model as some kind of universal kinetic equation for the evolution of a random particle described by position and velocity coordinates moving in $\mathbb{R}^d$ under the influence of an external force $\nabla \Phi$.

Motivated by our research and development of so called fiber lay-down models we are in need to derive specific manifold-valued analogues of Equation \eqref{Langevin_in _Rd} as well as spherical velocity models out of the latter. Herein, fiber lay-down processes arise in the production process of nonwovens and the expression is used for the description of the forms generated by the stochastic lay-down of flexible fibers onto a moving conveyor belt. The understanding and mathematical simulation of such fiber webs is of great industrial interest, see \cite{KMW09} and references therein. A simplified stochastic model simulating a virtual fiber in a fast and efficient way is developed in \cite{GKMW07} and describes the lay-down of a single fiber as a curve in $\mathbb{R}^2$. More realistic three-dimensional models are derived first in \cite{KMW12}. Consider \cite{GKMS12} for their mathematical analysis.

As explained in \cite{GKMS12} the two- and three-dimensional model can at once be formulated with the help of the manifold-valued Stratonovich stochastic differential equation (SDE) with state space $\mathbb{R}^d \times \mathbb{S}^{d-1}$, $d \in \mathbb{N}$, $d \geq 2$, given by
\begin{align} \label{Fiber_basic_introduction}
&\mathrm{d}\xi_t = \omega_t \, \mathrm{dt} \\
&\mathrm{d}\omega_t = - (I-\omega_t \otimes \omega_t) \, \nabla \Phi ( \xi_t) \, \mathrm{dt} + \sigma \, (I-\omega_t \otimes \omega_t) \circ \mathrm{d}W_t. \nonumber
\end{align}
Here $\mathbb{S}^{d-1}$ is the unit sphere in $\mathbb{R}^d$ and $x \otimes y = x y^T$. All basic notations as well as a short introduction to the concept of Stratonovich SDEs on manifolds is presented in Section \ref{Section_Setup}. In Section \ref{Section_Geometry_of_Fiber_Lay_Down} we start lighting up the geometry occuring behind this basic fiber lay-down model. We give a new interpretation and will see that this equation is just the natural analogue of the classical Langevin Equation \eqref{Langevin_in _Rd} having spherical velocities. 

The stochastics occuring in the fiber Equation \eqref{Fiber_basic_introduction} are modeled with the help of a Brownian motion on $\mathbb{S}^{d-1}$. Compared to realistic lay down-processes this gives to rough paths and one aims to have a model with smoother trajectories. Clearly, in $\mathbb{R}^d$ one may simply replace a standard Brownian motion through a classical Langevin equation with potential $\Phi=0$ to obtain a smoother stochastic process. In the spherical case, it is not a priori clear how to get a smoother version of some Brownian motion and how to formulate a spherical Langevin equation.

So this motivates to search for natural manifold-valued versions of the classical Langevin equation.  This is done in Section \ref{Derivation_Langevin_process} with the help of a simple transformation strategy that gives the general manifold-valued Stratonovich SDE, called the \textit{geometric Langevin equation}, in a surprisingly easy and natural way. The transformation strategy reads as follows. Via observing that the geometric Langevin equation must have state space equal to the phase space of some manifold $\mathbb{M}$, i.e., the tangent bundle $\mathbb{T}\mathbb{M}$, we first formulate the Langevin equation in its most natural way on the tangent bundle of $\mathbb{R} / {2 \pi \mathbb{Z}}$. This space is diffeomorphic to  $\mathbb{T}\mathbb{S}^1$. As consequence, we get a natural equation for the Langevin process on $\mathbb{T}\mathbb{S}^1$ which can immediately translated to $\mathbb{T}\mathbb{S}^d$ and afterwards finally even easily to $\mathbb{T}\mathbb{M}$. Here $\mathbb{M}$ may be any regular manifold in $\mathbb{R}^N$ of dimension $d$ as introduced in Section \ref{Section_Setup}. The equation with state space $\mathbb{T}\mathbb{M}$ then reads
\begin{align} \label{Langevin_equation_geometric_int}
&\mathrm{d} \xi_t = \omega_t \, \mathrm{dt} \\
&\mathrm{d} \omega_t =  - \lambda \,\omega_t \,\mathrm{dt}- F(\xi_t,\omega_t) \,\mathrm{dt} - \text{grad}_{\mathbb{M}} \, \Phi (\xi_t) \,\mathrm{dt} + \sigma \,\Pi_{\mathbb{M}}[\xi_t] \circ \mathrm{d}W_t. \nonumber
\end{align}
Here $\text{grad}_{\mathbb{M}}$ denotes the (tangential) gradient on $\mathbb{M}$, $\Pi_{\mathbb{M}}[\xi]$ the orthogonal projection onto the tangent space at the point $\xi \in \mathbb{M}$ and $F(\xi,\omega)$ is a suitable forcing term keeping the trajectories on $\mathbb{T}\mathbb{M}$. For details and notations, see Section \ref{Section_Setup} and Section \ref{Derivation_Langevin_process}. This equation has also been detected before by Lelièvre, Rousset and Stoltz, see \cite[Eq.~(3.3)]{LRS12} or \cite[Sec.~3.3]{LRS10}. Nevertheless, our strategy to obtain the natural generalization of the classical Langevin equation completely differs from the strategy given in \cite{LRS12} and moreover, exactly the same arguments afterwards serve as starting point for developing the above mentioned smoother version of the basic fiber lay-down model. We also strongly believe that the underlying strategy is helpful for every applied mathematician who aims to derive similiar manifold-valued stochastic kinetic equations. Therefore, we should present the approach in full detail.

Up to the authors best knowledge there are apparently not many mathematical papers available dealing with the construction and analysis of Langevin type equations in its general geometric form, see Section \ref{Langevin_Generator}. We found the surprisingly and seemingly unknown paper of Soloveitchik, see \cite{Sol95}, in which the author constructs the Langevin process (called Ornstein-Uhlenbeck process therein) with external potential on smooth, compact and connected Riemannian manifolds. First, Soloveitchik introduces a natural generalization of the classical Langevin generator and constructs afterwards its diffusion process. In this section we connect our geometric Langevin equation with Soloveitchiks approach via proving that the generator $L$ associated to \eqref{Langevin_equation_geometric_int} coincides with the one proposed by Soloveitchik, i.e., for some local coordinate system $(q^1,\ldots,q^d,v^1,\ldots,v^d)$ on $\mathbb{T}\mathbb{M}$ we obtain
\begin{align*}
L=\sum \left( v^j \frac{\partial}{\partial q^j} -  \Gamma^j_{nm}  v^n v^m \frac{\partial}{\partial v^j} - g^{ij}\frac{\partial \Phi}{\partial q^i} \frac{\partial}{\partial v^j} - \lambda \, v^{j} \frac{\partial}{\partial v^j} \right) + \frac{1}{2} \sigma^2 \sum_{i,j} g^{ij} \frac{\partial^2}{\partial v^i \partial v^j}.
\end{align*}
Having in mind our applications, such local generator representations then directly lead to convenient numerical simulation methods for equation \eqref{Langevin_equation_geometric_int} via formulating so called local SDEs, see Remark \ref{Rm_simulating_geometric_Langevin}. Moreover, let us already remark the article \cite{Jor78} by J{\o}rgensen in which the same operator appears without forcing term as well as the article \cite{Kol00} in which Kolokoltsov even discusses more general versions of the operator from above.

In view of our applications we restrict attention afterwards to the spherical situation and discuss the spherical Langevin procss on $\mathbb{T}\mathbb{S}^2$ seperately. This includes explicit simulation formulas and is done in Section \ref{Spherical_Langevin}.

In Section \ref{Section_Applications} we apply the previous machinery to develop the smooth fiber lay-down model. The driving stochastics therein is now given by the desired spherical Langevin process. The general equation has state space $\mathbb{R}^d \times \mathbb{T} \mathbb{S}^{d-1}$, $d \in \mathbb{N}$, $d \geq 2$, and reads
\begin{align} \label{smooth_fiber_introduction}
&\mathrm{d}\xi_t = \omega_t \, \mathrm{dt} \\
&\mathrm{d}\omega_t = - (I-\omega_t \otimes \omega_t)  \nabla \Phi ( \xi_t) \, \mathrm{dt} + \mu_t \, \mathrm{dt} \nonumber \\
&\mathrm{d}\mu_t =  \mu_t \cdot \nabla \Phi(\xi_t) ~ \omega_t  \,\mathrm{dt} - \lambda \, \mu_t \,\mathrm{dt} -|\mu_t|^2 \omega_t  \, \mathrm{dt} + \sigma \, (I-\omega_t \otimes \omega_t) \circ \mathrm{d}W_t. \nonumber
\end{align}
The physical relevant cases are $d=2$ and $d=3$. This model is developed in cooperation with our colleagues Klar, Maringer and Wegener, i.e., the authors from \cite{KMW12B}. The equation was first detected in its full form including all necessary terms by the second named author of the underlying article. This important fact is not mentioned in \cite{KMW12B}. Again we provide simple local coordinate expressions guaranteeing convenient numerical simulations of Equation \eqref{smooth_fiber_introduction}, see Proposition \ref{Pp_local_SDE_smooth_model}. A detailed comparision between the basic and the smooth fiber lay-down model and its industrial applications will be discussed in a forthcoming research paper of the authors from \cite{KMW12B} together with the authors of the present article.

In Section \ref{Section_Applications} we present further applications to some spherical velocity models arising in self-propelled interacting systems describing the collective behaviour of swarms of animals, see \cite{CKMT10}. They have a similiar mathematical structure as the two-dimensional fiber lay-down equation. Analogously as in the fiber lay-down scenario they can now directly be translated into the three-dimensional case and moreover, while substituting a spherical Brownian motion through a spherical Langevin process, new smooth versions of the models can be formulated.

Motivated by the basic fiber lay-down model, we derive in Section \ref{Langevin_process_constant_velocity_version} a natural version of the geometric Langevin equation moving with velocity of constant absolute value. This equation is then called the \textit{geometric Langevin equation with spherical velocity}. It lives on  $\mathbb{S}_r\mathbb{M}$, the spherical tangent bundle of $\mathbb{M}$ of radius $r > 0$, and is given by
\begin{align} \label{Langevin_equation_geometric_constant_speed_final_intro}
&\mathrm{d} \xi_t = \omega_t \, \mathrm{dt} \\
&\mathrm{d} \omega_t =  - F(\xi_t,\omega_t) \, \mathrm{dt}   - \Big(\Pi_{\mathbb{S}_r}[\omega_t] \Pi_{\mathbb{M}}[\xi_t] \Big)\nabla \Phi (\xi_t) \,\mathrm{dt} + \sigma \, \Big(\Pi_{\mathbb{S}_r}[\omega_t]  \,\Pi_{\mathbb{M}}[\xi_t] \Big)\circ \mathrm{d}W_t. \nonumber
\end{align}
Up to the authors best knowledge, this equation seems to be new and has nowhere proposed before. This equation contains in case $\mathbb{M}=\mathbb{R}^d$ the basic fiber lay-down model as special case. Moreover, there are now a whole galaxy of new beautiful and powerful mathematical models available.

The reader can see the beauty of the spherical velocity version of the geometric Langevin equation in Section \ref{Examples_My_geometric_Langevin_equation} where we discuss specific examples. One example deals again with the spherical situation. Therein, the spherical velocity version of the geometric Langevin process on $\mathbb{S}_r\mathbb{M}$, $\mathbb{M}=\mathbb{S}^2$, is constructed and a special coordinate representation is derived. Altogether, this gives a further smooth spherical process which can again be used to construct a new smooth three-dimensional fiber lay-down model. We expect that the latter serves as real alternative to all existing three-dimensional lay-down models. An additional research article on this application is intended.

We remark that the discussion about the geometry of the basic fiber lay-down in Section \ref{Section_Geometry_of_Fiber_Lay_Down} serves partly as motivation. The reader who is not interested in this application may directly switch to Section \ref{Derivation_Langevin_process}. Furthermore, in Section \ref{Section_Setup} we present the necessary mathematical background and introduce notations that are used until the end of this article.

Finally, we believe that this article is interesting from both a theoretical and applied point of view. Sometimes, pure and applied mathematics are seemingly disjoint. With the help of this article, we also hope to build a small bridge connecting some real world applications with the more theoretical area of SDEs on manifolds.

\section{The setup and notations} \label{Section_Setup}

Before starting with the geometric derivation and discussion of Langevin type equations on general submanifolds in the euclidean space, we introduce some notations and recall well-known facts. We stay detailed in order to be self-contained. First some general notations: $C^\infty(\mathbb{R}^n)$, $n \in \mathbb{N}$, denotes the set of all infinitely often differentiable functions $f: \mathbb{R}^n \rightarrow \mathbb{R}$. The index $c$ means compact support.  $\nabla$ (or $\nabla_x$) always denotes the usual gradient operator (with respect to the variable $x$)  as column vector and $\nabla^2$ the Hessian matrix in the euclidean space. $| \cdot|$ is the standard euclidean norm. The standard euclidean scalar product is simply denoted by $\cdot$ or also by $\left( \cdot, \cdot \right)_{\text{euc}}$. Superscript $T$ denotes the transpose and rank the usual rank of some matrix. The expression smooth means that the underlying object is of class $C^\infty$. $I$ is the identity matrix. Partial derivatives with respect to some variable $x$ are denoted as usual by $\frac{\partial}{\partial x}$ or for short by $\partial_x$. Convention: Any vector $x \in \mathbb{R}^n$ is always understood as column vector. And the notation $(x,y)$ for $x \in \mathbb{R}^n$, $y \in \mathbb{R}^m$ is understood as 
\begin{align*}
(x,y)=\begin{pmatrix} x\\y \end{pmatrix} \in \mathbb{R}^{n+m}.
\end{align*}

Until the end of this article we follow the notations and language introduced in the underlying section without further mention this again. 

\subsection{Submanifolds in the euclidean space} \label{df_and_not}
In the following, let $d,k \in \mathbb{N}$ and define $N:=d+k$. We fix some $d$-dimensional (regular) submanifold $\mathbb{M} \subset \mathbb{R}^{N}$ of the form 
\begin{align} \label{rank_condition}
\mathbb{M}= \left\{ \xi \in \mathbb{R}^{N}~\Big|~f_1(\xi)=0,~\ldots,f_k(\xi)=0 \right\},~\text{rank} \,J_\xi(f_1,\ldots,f_k)=k,~\xi \in \mathbb{M}.
\end{align}
Here $f_1,\ldots,f_k \in C^\infty(\mathbb{R}^{N})$ and $J_\xi := J_\xi(f_1,\ldots,f_k):\mathbb{R}^N \rightarrow \mathbb{R}^k$ denotes the Jacobian matrix of $f_1,\ldots,f_k$ at the point $\xi \in \mathbb{M}$, i.e.,
\begin{align*}
J_\xi(f_1,\ldots,f_k)= \begin{pmatrix} \frac{\partial f_1}{\partial \xi_1}(\xi) & \ldots & \frac{\partial f_1}{\partial \xi_{N}}(z) \\ \vdots & & \vdots \\ \frac{\partial f_k}{\partial \xi_1}(\xi) & \ldots & \frac{\partial f_k}{\partial \xi_{N}}(\xi) \end{pmatrix},~\xi \in \mathbb{M}.
\end{align*}

In case $k=1$, we write $\mathbb{H} := \mathbb{M}$. In this situation we introduce the following (orientated) unit normal vector field $n$ on $\mathbb{H}$ as
\begin{align*}
n(\xi) :=  \frac{1}{\left| \nabla f_1 (\xi) \right|} \nabla f_1(\xi),~\xi \in \mathbb{H}.
\end{align*}

Coming back to the general situation, the tangent space $T_\xi\mathbb{M}$ at the point $\xi \in \mathbb{M}$, embedded in $\mathbb{R}^{N}$, is now given by
\begin{align*}
T_\xi\mathbb{M} = \left\{ \omega\in \mathbb{R}^{N}~\Big|~ J_\xi \,\omega=0 \right\} = \left\{ \omega \in \mathbb{R}^{N}~\Big|~ \omega \cdot \nabla f_j (\xi) =0,~j=1,\ldots,k \right\} 
\end{align*}

We identify each element $\omega$ from the embedded tangent space $T_\xi M \subset \mathbb{R}^{N}$, $\xi \in M$, with its associated derivation $\omega\cdot \nabla(\xi)$ from the algebraic tangent space, in notation 
\begin{align*}
\omega  \equiv   \omega \cdot \nabla(\xi).
\end{align*}
Then $\omega(f)= \omega \cdot \nabla f (\xi)$, $f \in C^\infty(U)$, $U \subset M$ open with $\xi \in U$, is understood as $\omega \cdot \nabla g(\xi)$ where $g \in C^\infty(\mathbb{R}^{N})$ is a function which extends $f$ locally in some open neighbourhood of $M$ at $\xi$. Recall that $\omega \cdot \nabla g (\xi)$ is independent of each such extension $g$ for $f$. 

Now the tangent bundle of $\mathbb{M}$, i.e., $\mathbb{T}\mathbb{M} = \coprod_{\xi \in \mathbb{M}} T_\xi \mathbb{M}$, is given as embedded manifold in $\mathbb{R}^{2N}$ simply by
\begin{align*}
\mathbb{T} \mathbb{M}= \left\{ (\xi,\omega) \in \mathbb{R}^{2N}~\Big|~ f_j(\xi)=0,~\omega \cdot \nabla f_j(\xi)=0,~j=1,\ldots,k \right\}.
\end{align*}
So $\mathbb{T}\mathbb{M}$ is again a (regular) submanifold of $\mathbb{R}^{2N}$ of dimension $2d$ since with $h_j(\xi,\omega) := f_j(\xi)$ for $j=1,\ldots,k,$ and with $h_j(\xi,\omega) := \omega \cdot \nabla f_j (\xi)$ for $j=k+1,\ldots,2k,$ we have for the Jacobian
\begin{align*}
J(h_1,\ldots,h_{2k})  = \begin{pmatrix} \nabla f_1 & \ldots & \nabla f_k & \nabla^2 f_1 \, \omega  & \ldots & \nabla^2 f_k   \, \omega \\
                                           0  & \ldots & 0  & \nabla f_1 &  \ldots & \nabla f_k \end{pmatrix}^T.
\end{align*}
Here $\nabla^2 f_j \, \omega$ is understood as the mapping $\mathbb{R}^{2N} \ni (\xi,\omega) \mapsto \nabla^2 f_j (\xi) \, \omega \in \mathbb{R}^N$ for all $j=1\ldots,k$. Consequently, we get
\begin{align*}
\text{rank}\, J_{(\xi,\omega)} (h_1,\ldots,h_{2k})  = 2k,~(\xi,\omega) \in \mathbb{T}\mathbb{M}.
\end{align*}
Furthermore, we define the spherical tangent bundle bundle $\mathbb{S}_r\mathbb{M}$ of radius $r$ as
\begin{align*}
\mathbb{S}_r\mathbb{M} = \left\{ (\xi,\omega) \in \mathbb{T} \mathbb{M}~|~ |\omega|^2=r^2 \right\},~r >0.
\end{align*}
In case $r=1$ we shortly write $\mathbb{S} \mathbb{M}$ and call it the unit tangent bundle. The function $h_{2k+1}(\xi,\omega) := \sum_{j=1}^N \omega_j^2 - r^2$ serves as an additional defining function for $\mathbb{S}_r\mathbb{M}$. Then it is easy to see that $\mathbb{S}_r\mathbb{M}$ is a (regular) submanifold of $\mathbb{R}^{2N}$ of dimension $2d-1$. $\mathbb{T}  \mathbb{M}$ and $\mathbb{S}_r\mathbb{M}$ will later on serve as the right state spaces for generalized Langevin type equations. 

Now choose a smooth mapping $\mathbb{T}\mathbb{M} \ni (\xi,\omega) \mapsto \mathcal{A}(\xi,\omega) \in \mathbb{R}^{2N}$. Then $\mathcal{A}$ is a smooth vector field on $\mathbb{T}\mathbb{M}$, i.e., it holds additionally $\mathcal{A}(\xi,\omega) \in T_{(\xi,\omega)} \mathbb{T}\mathbb{M}$ for any $(\xi,\omega) \in \mathbb{T} \mathbb{M}$, if and only if
\begin{align} \label{tangent_condition_tangentbundle}
\begin{pmatrix} \nabla f_j  \\ 0 \end{pmatrix} \cdot \mathcal{A}  =0,~ \begin{pmatrix} \nabla^2 f_j \, \omega \\ \nabla f_j  \end{pmatrix} \cdot \mathcal{A} =0,~j=1,\ldots,k.
\end{align}
Moreover, $\mathcal{A}$ is even a smooth vector field on $\mathbb{S}_r\mathbb{M}$ if additionally to \eqref{tangent_condition_tangentbundle} we have
\begin{align} \label{tangent_condition_unittangentbundle}
\begin{pmatrix} 0 \\ \omega \end{pmatrix} \cdot \mathcal{A}(\xi,\omega)=0,~(\xi,\omega) \in \mathbb{S}_r\mathbb{M}.
\end{align}

Further recall that for some local coordinate chart $(U,q)$ of $\mathbb{M}$ with $q:U \rightarrow q(U)$, $q(U) \subset \mathbb{R}^d$, the vectors
\begin{align*}
\frac{\partial \tau}{\partial s_1}(s),\ldots, \frac{\partial \tau}{\partial s_d}(s)
\end{align*}
form a basis of $T_\xi \mathbb{M}$, $\xi \in \mathbb{M}$. Here $\tau$ is defined as $\tau := q^{-1}$ on $q(U)$ and $s \in q(U)$ is chosen such that $\tau(s) = \xi$. Let $q^j$, $j=1,\ldots,d$, denote the coordinate functions of $q$. Moreover, the vector field $\frac{\partial}{\partial s_j}$ from $q(U)$ corresponds under the local diffeomorphism $q$ to the vector field $\frac{\partial \tau}{\partial s_j}$ on $U$, $j=1,\ldots,d$. More precisely, let $\widetilde{\frac{\partial}{\partial s_j}}$ be defined by 
\begin{align*}
\widetilde{\frac{\partial}{\partial s_j}} \widetilde{f} (\xi) := \frac{\partial}{\partial q_j} \widetilde{f} (\xi):=  \frac{\partial}{\partial s_j} f (s),~j=1,\ldots,d,
\end{align*}
where $f \in C^\infty(q(U))$ and $\widetilde{f} \in C^\infty(U)$ are related by $\widetilde{f} \circ \tau =f$. Here $\frac{\partial}{\partial q_j}$ is the usual algebraic derivation induced by the coordinate functions. Hence $\widetilde{\frac{\partial}{\partial s_j}}\, \equiv  \, \frac{\partial \tau}{\partial s_j}  $ since
\begin{align} \label{eq_computing_push_forward}
\widetilde{\frac{\partial}{\partial s_j}} \widetilde{f} (\xi) = \frac{\partial}{\partial s_j} \left( \widetilde{f} \circ \tau \right) (s) = \frac{\partial \tau}{\partial s_j} (s) \cdot \nabla \widetilde{f} (\xi).
\end{align}
In the same way as in \eqref{eq_computing_push_forward} one computes corresponding vector fields under some diffeomorphism between two manifolds. 

The corresponding local coordinate chart $(TU, Tq )$ of $\mathbb{T}\mathbb{M}$ is now determined through $TU= \{ (\xi,\omega) \in \mathbb{T} \mathbb{M}~|~\xi \in U\}$ and
\begin{align*}
(Tq): TU \rightarrow q(U) \times \mathbb{R}^d,~(\xi,\omega) \mapsto (q^1(\xi),\ldots,q^d(\xi),v^1(\xi,\omega),\ldots,v^d(\xi,\omega))
\end{align*}
where the $v^j$ are uniquely determined by the relation $\omega= \sum_{j=1}^d v^j(\xi,\omega) \, \frac{\partial \tau}{\partial s_j} (\xi)$. In other words, we have the diffeomorphism
\begin{align} \label{eq_diffeomorphism_chart}
(Tq)^{-1}: q(U) \times \mathbb{R}^d \rightarrow TU,~(s,\kappa) \mapsto (\xi(s),\omega(s,\kappa)) := \left( \tau(s), \sum_{j=1}^d\kappa_j \, \frac{\partial \tau}{\partial s_j}(s) \right)
\end{align}
where $s=(s_1,\ldots,s_d) \in q(U)$ and $\kappa=(\kappa_1,\ldots,\kappa_d) \in \mathbb{R}^d$. Furthermore, $\mathbb{M}$ inherits canonically the structure of a Riemannian manifold obtained from the underlying euclidean space $\mathbb{R}^{N}$. Thus the Riemannian metric is determined in the coordinate chart $(U,q)$ via $g_{ij}= \frac{\partial \tau}{\partial s_i} \cdot \frac{\partial \tau}{\partial s_j}$ for $i,j=1,\ldots,d$. The inverse matrix associated to $(g_{ij})$ is denoted as usual by $(g^{ij})$. The Christoffel symbols are given by the formula 
\begin{align*}
\Gamma_{nm}^{i} = \frac{1}{2} \sum_{j} g^{ij} \left( \frac{\partial g_{jn}}{\partial s_m} + \frac{\partial g_{jm}}{\partial s_n} - \frac{\partial g_{nm}}{\partial s_j} \right),~i,n,m=1,\ldots, d.
\end{align*}
Through using the previous relation for the $g_{ij}$ one immediately verifies 
\begin{align} \label{eq_relation_Christoffel symbols}
\Gamma_{nm}^{i} =  \sum_{j} \, g^{ij}\, \frac{\partial \tau}{\partial s_j} \cdot \frac{\partial^2 \tau}{\partial s_n \partial s_m},~i,n,m=1,\ldots, d.
\end{align}

Next we introduce the orthogonal projection from $\mathbb{R}^{N}$ onto $T_\xi \mathbb{M}$, $\xi \in \mathbb{M}$, as
\begin{align} \label{Eq_orthogonal_Projection}
\Pi_{\mathbb{M}}[\xi] = I - J_\xi^{\,T} \left(J_\xi \, J_\xi^{\,T}\right)^{-1} J_\xi.
\end{align}
By the rank condition \eqref{rank_condition}, note that $\left(J_\xi \, J_\xi^{\,T}\right)^{-1}$ really exists and is smooth as consequence of the powerful implicit function theorem. In case $\mathbb{M}=\mathbb{H}$, \eqref{Eq_orthogonal_Projection} reduces to $\Pi_{\mathbb{H}}  = I- n \otimes n$. Here $x \otimes y := x y^T$, $x,y \in \mathbb{R}^{N}$. Finally, the (tangential) gradient of some $f \in C^\infty(\mathbb{M})$ reads
\begin{align} \label{Def_gradient_submanifold}
\text{grad}_{\mathbb{M}} \,f(\xi) = \Pi_{\mathbb{M}}[\xi] \, \nabla f (\xi),~\xi \in \mathbb{M}.
\end{align}
The latter is again understood as applied to some smooth local extension $g$ (in $\mathbb{R}^{N}$) of $f$ around $\xi$. This is indeed well-defined. 

\begin{Rm}
In order to include $\mathbb{R}^d$ itself into this framwork, we regard $\mathbb{R}^d$ as canonically embedded in $\mathbb{R}^{d+1}$ via
\begin{align} \label{identification_Rd}
\mathbb{R}^d = \{ \xi \in \mathbb{R}^{d+1}~|~ \xi_{d+1}=0\}.
\end{align}
Then note that $\mathbb{S}_r\mathbb{M}$, $\mathbb{M}:=\mathbb{R}^d$, can  canonically be identified with $\mathbb{R}^d \times \mathbb{S}_r^{d-1}$ with $\mathbb{S}_r^{d-1}$ the sphere of radius $r$ in $\mathbb{R}^d$.
\end{Rm}

Finally, we need some basics about Stratonovich stochastic differential equations (abbreviated by SDEs) on manifolds discussed next.

\subsection{Stratonovich SDEs on manifolds} \label{Appendix_Strat_manifolds}
We shortly discuss the concept of manifold-valued Stratonovich SDEs. Excellent mathematical treatments on this subject can e.g.~be found in \cite{Hsu02}, \cite{HT94} or \cite{IW89}. So let $\mathbb{X}$ be a general (abstract) smooth manifold and let $(\Omega, \mathcal{F}, \mathbb{P}, \{\mathcal{F}_t\}_{t \geq 0})$ be a standard filtered probability space equipped with an $r$-dimensional standard $\{\mathcal{F}_t\}_{t \geq 0}$\,-\,Brownian motion $W=\{W_t\}_{t \geq 0}$. Let $\mathcal{V}_0,~\mathcal{V}_1, \ldots, \mathcal{V}_r$ be smooth vector fields on $\mathbb{X}$ and let $x \in \mathbb{X}$. A solution $X=\{X_t\}_{t \geq 0}$ of the Stratonovich stochastic differential equation
\begin{align} \label{Df_Stratonovich_equation}
\mathrm{d}X_t = \mathcal{V}_0 (X_t) \,\mathrm{dt} + \sum_{j=1}^r \mathcal{V}_j (X_t) \circ \mathrm{d}W_t^{(j)}
\end{align}
with initial condition $X_0=x$ is any $\{\mathcal{F}_t\}_{t \geq 0}$\,-\,adapted, continuous process on $\hat{\mathbb{X}}$ (the one point-compactification of $\mathbb{X}$) having $\Delta$ as a trap such that the following is satisfied: $X_0=x$ holds $\mathbb{P}$-a.s.~and for every $f \in C_{c}^\infty(\mathbb{X})$ the process $\{f(X_t)\}_{t\geq 0}$ satisfies the $\mathbb{R}$-valued Stratonovich SDE
\begin{align*}
\mathrm{d} f(X_t) = (\mathcal{V}_0 f) (X_t) \,\mathrm{dt} +  \sum_{j=1}^r (\mathcal{V}_j f) (X_s) \circ \mathrm{d}W_t^{(j)}.
\end{align*}
Any such $X$ has generator $L:C^\infty(\mathbb{X}) \rightarrow C^\infty(\mathbb{X})$, $L=\mathcal{V}_0 + \frac{1}{2} \sum_{j=1}^r \mathcal{V}_j^2$, and is a $L$-diffusion process, see \cite[Sec.~1.3]{Hsu02}. Such $L$-diffusion processes are weakly unique in the sense that they induce a unique $L$-diffusion measure on the path space, see \cite[Sec.~1.3]{Hsu02}. Moreover, the Stratonovich solution concept behaves in a natural way under state space transformations: If $\mathbb{X}$ is diffeomorphic to some $\mathbb{Y}$ with diffeomorphsim $\varphi:\mathbb{X} \rightarrow \mathbb{Y}$, then $Y=\varphi(X)$ solves the associated Stratonovich SDE on $\mathbb{Y}$ in which the vector fields $\mathcal{V}_i$ from $\mathbb{X}$ are replaced through their push-forward vector fields under $\varphi$, see \cite[Prop.~1.2.4]{Hsu02}. We call the SDEs on $\mathbb{X}$ and $\mathbb{Y}$ equivalent.\\ Finally, if $\mathbb{X} \subset \mathbb{R}^n$ is a regular submanifold as defined in the previous section with embedded vector fields $\mathcal{V}_j$ then SDE \eqref{Df_Stratonovich_equation} can equivalently be viewed as some usual Stratonovich SDE in $\mathbb{R}^n$ in which the vector fields $\mathcal{V}_j$ are arbitrarily extended to smooth vector fields on $\mathbb{R}^n$ (denoted again by $\mathcal{V}_j$). This means that the solution to \eqref{Df_Stratonovich_equation} is obtained by solving the extended equation in $\mathbb{R}^n$. The solution to the extended equation then stays on $\mathbb{X}$ provided that the initial value lies on $\mathbb{X}$, see \cite[Prop.~1.2.8]{Hsu02}. In this case we introduce some notations. Define the matrix valued mapping $p \mapsto \mathcal{V}(p)=(\mathcal{V}_1(p),\ldots,\mathcal{V}_r(p))$, $p \in \mathbb{R}^n$, and set
\begin{align*}
\mathcal{V} \circ \mathrm{d}W_t:=\sum_{j=1}^r (\mathcal{V}_j f) (X_s) \circ \mathrm{d}W_t^{(j)}.
\end{align*}

\section{Starting point: The geometry of fiber lay-down and a new modeling point of view} \label{Section_Geometry_of_Fiber_Lay_Down}

As already mentioned in the introduction, this section serves partly as motivation. The reader who is probably not familiar with the fiber lay-down model may skip this section in first reading. Nevertheless, due to didactic reasons, we start in this section with the discussion of the geometry occuring behind the so called (basic) fiber lay-down equations. In particular, we give a new geometric modeling of the latters. We recall that all of these equations are simplified stochastic models simulating a virtual fiber web as observed and needed for optimization in the production process of nonwovens. For further motivation we refer to the description from the introduction and to \cite{KMW09}. 

Several mathematical questions arise out of the fiber lay-down geometry dealing essentially with questions about geometric Langevin type equations and serve as motivation for all our further studies in the rest of the underlying article. As described in the introduction, the aim of developing a smoother version of the basic fiber lay-down model yields us to the study of manifold-valued versions of the classical Langevin equation. Moreover, we will see that the basic fiber lay-down model, developed originally in \cite{GKMW07} and generalized in \cite{KMW12} and \cite{GKMS12} to higher dimensions, is just a special realization of a geometric equation which will later on be called the geometric Langevin equation with spherical velocity, see Section \ref{Langevin_process_constant_velocity_version}. 

But first, we start by revisiting the original modeling of fiber lay-down.

\subsection{The arc-length modeling of fiber lay-down revisited}

The two-dimensional fiber lay-down equation in case of a non-moving conveyor belt is originally introduced in \cite{GKMW07}. The lay-down of a single fiber is modeled therein as a curve  $\xi$ in $\mathbb{R}^2$. The basic assumption and starting point of the authors is that the curve is assumed to be arc-length parametrized. Consequently, one has $| \partial_t \xi_t |=1$ where $t$ has just the interpretation of arc-length. With the ansatz $\partial_t \xi_t =(\cos(\alpha_t),\sin(\alpha_t))^T$, the authors of \cite{GKMW07} propose a stochastic equation in $\alpha$ in order to get a realistic fiber lay-down behaviour given by the arc-length parametrized curve $\xi_t$, $t \geq 0$. Furthermore, in all upcoming articles and models, see e.g.~\cite{BGKMW07},~\cite{KMW09},\cite{KMW12} and \cite{KMW12B}, the interpretation of $t$ being the arc-length is still valid. Consequently, considering this point of view, the models do not have a priori a physical interpretation.

From a physical point of view, the evolution equation for the fiber lay-down process must be obtained in a natural way with $t$ really being time. Henceforth, let us propose an alternative point of view which is even fundamental apart from the fiber lay-down discussion, see Section \ref{Langevin_process_constant_velocity_version}. 

\subsection{The basic fiber lay-down equation: A second view by view and a new modeling approach} \label{new_point_fiber_modeling}

Therefore, we consider the following scenario in the production process of nonwovens (see \cite{BGKMW07}). We also assume the conveyor belt to be stationary. The spinning speed of a single fiber, i.e., the amount of (fiber) material coming out of a single nozzle, is assumed to be a constant equal to $v_{s} \in (0,\infty)$. Here the fiber may be seen as an infinitesimal thin one-dimensional, elastic, slender object. Chosing the time unit appropriately we may assume without loss of generality that $v_s=1$. The fiber lay-down process now draws a curve $\xi_t \in \mathbb{R}^{d},~t \geq 0$ the time index, on the stationary belt. Here either $d=2$ or $d=3$ are treating the physical relevant scenarios. Due to an turbulent flow in the deposition region near the conveyor belt, the behaviour of the curve is of stochastic nature. By the previous it is now reasonable to assume that the lay-down speed of the fiber (material) again coincides with its spinning speed $v_s=1$. This includes that we assume the fiber to be inextensible in the non-moving conveyor belt case. So, if $\omega_t$ denotes the attached velocity vector at the lay-down curve at time $t \geq 0$, it must hold $|\omega_t|=1$. Thus the tuple $(\xi_t,\omega_t)$ lives on the unit tangent bundle of $\mathbb{R}^d$ for all $t \geq 0$. The latter is equal to $\mathbb{R}^{d} \times \mathbb{S}^{d-1}$ where $\mathbb{S}^{d-1}$ denotes the unit sphere in $\mathbb{R}^d$. Moreover, in reality one has an ergodic behaviour of the fiber lay-down process (in case of a stationary belt) and observes that the distribution of $\xi_t$ converges towards a unique stationary state as $t \rightarrow \infty$. 

Now consider a particle which is governed by the classical Langevin equation \eqref{Langevin_in _Rd}. Then its density in phase space $\mathbb{R}^{2d}$ fulfills the associated Fokker-Planck evolution equation. A stationary solution to the latter is $e^{-\Phi(x)}e^{- \frac{\lambda}{\sigma^2} v^2}$. Thus one expects that the stationary distribution of the particle is given up to normalization by $e^{-\Phi(x)}e^{- \frac{\lambda}{\sigma^2} v^2} \mathrm{d}x \mathrm{d}v$. Because of this we regard the Langevin SDE \eqref{Langevin_in _Rd} as some kind of universal stochastic kinetic model for an ergodic process described by position and velocity coordinates.

By the previous discussion we propose the most simpliest ansatz for the ergodic stochastic fiber lay-down process from above in form of an usual Langevin SDE \eqref{Langevin_in _Rd} in $\mathbb{R}^{2d}$  for $(\xi_t,\omega_t)$, $t \geq 0$. The potential $\Phi \in C^\infty(\mathbb{R}^d)$ then has to be determined later on. The further restriction $|\omega_t|=1$, $t \geq 0$, suggests to do a suitable spherical projection of the velocity coordinates $\omega$ in this equation. Intuitively, one has to project an infinitesimal step of the $\omega_t$-component in \eqref{Langevin_in _Rd} onto the sphere. This is analogous to the construction of spherical Brownian motion, see \cite[Ch.~5,~page~183]{RW87}. This projection of course has to be be done with the Stratonovich operation which obeys the natural transformation rules from ordinary calculus, see e.g.~\cite{Hsu02}. In other words, the describing Stratonovich equation for the easiest case of the fiber lay-down process with state space $\mathbb{R}^d \times \mathbb{S}^{d-1}$ simply reads 
\begin{align*}
&\mathrm{d}\xi_t = \omega_t \, \mathrm{dt} \\
&\mathrm{d}\omega_t = - (I-\omega_t \otimes \omega_t) \left(\lambda \omega_t + \nabla_\xi \Phi ( \xi_t)\right)  \mathrm{dt} + \sigma \, (I-\omega_t \otimes \omega_t) \circ \mathrm{d}W_t.
\end{align*}
Here $W$ is standard $d$-dimensional Brownian motion and $\lambda, \sigma$ are both nonnegative constants. Furthermore, recall the notation $(I-\omega \otimes \omega) \circ dW_t$ introduced in Section \ref{Appendix_Strat_manifolds}. The solution to this equation stays on $\mathbb{R}^d \times \mathbb{S}^{d-1}$ provided that the initial value lies on the latter manifold, see \cite[Prop.~1.2.8]{Hsu02}. Also note that $(I-\omega \otimes \omega)\omega = 0$ for $\omega \in \mathbb{S}^{d-1}$. Consequently, the equation reduces to
\begin{align} \label{SDE_Basic_Fiber_Lay_Down}
&\mathrm{d}\xi_t = \omega_t \, \mathrm{dt} \\
&\mathrm{d}\omega_t = - (I-\omega_t \otimes \omega_t)  \nabla_\xi \Phi ( \xi_t) \, \mathrm{dt} + \sigma \, (I-\omega_t \otimes \omega_t) \circ \mathrm{d}W_t \nonumber
\end{align}
with state space $\mathbb{R}^d \times \mathbb{S}^{d-1}$. This is exactly the basic two-dimensional (for $d=2$) and three-dimensional (for $d=3$) fiber lay-down equation, see \cite{GKMS12} for details and explicit simulation formulas. For the original derivation consider \cite{KMW12}. But now $t$ is really the time and the equation comes out in a completely easy and natural way, without any calculation. As shown in \cite{GKMS12}, a stationary distribution to the latter equation is given by $e^{-(d-1)\Phi} d\xi \otimes \mathcal{S}$, $d \in \mathbb{N}$, $d \geq 2$, where $\mathcal{S}$ denotes the Riemannian volume measure on $\mathbb{S}^{d-1}$ and $d\xi$ the usual Lebesgue measure in $\mathbb{R}^d$. Here $\Phi$ is usually chosen as $\Phi(\xi)=\sigma_1 \xi_1^2 + \sigma_2 \xi_2^2 + \sigma_3 \xi_3^2$, $\sigma_i > 0$, ensuring that the lay-down curve comes back to its reference point $0$ determined by the nozzle. 

Finally, we remark that the curve $\gamma$ which describes the lay-down in case of a non-stationary conveyor belt with speed $v_b \in [0,1]$ moving in direction $e$, $|e|=1$, may now also naturally be modeled as $\gamma_t=\xi_t + v_b \,t\, e$, $t \geq 0$. This is a slightly different and alternative modeling suggestion as considered by the arc-length point of view in \cite{BGKMW07}.

Summarizing, apart from the fiber lay-down scenario, the discussion altogether shows that Equation \eqref{SDE_Basic_Fiber_Lay_Down} may be interpreted as some kind of a universal stochastic kinetic model for an ergodic process described by position and velocity coordinates moving in $\mathbb{R}^d$ with velocity of constant absolute value. This seems to be fundamental and interesting by itself and in Section \ref{Langevin_process_constant_velocity_version} we use and generalize the previous projection strategy once more.

\section{Derivation of the geometric Langevin process on regular submanifolds and its Stratonovich SDE} \label{Derivation_Langevin_process}

As starting point, let us consider the Langevin type equation
\begin{align} \label{Langevin_equation_Rn}
&\mathrm{d} x_t = v_t \, \mathrm{dt} \\
&\mathrm{d} v_t = -\lambda \, v_t\,\mathrm{dt} - \nabla_x \Psi(x_t)\,\mathrm{dt} + \sigma \circ \mathrm{d} W_t \nonumber
\end{align} 
with state space $\mathbb{R}^{2d}$, $d \in \mathbb{N}$. Here $\Psi \in C^\infty(\mathbb{R}^d)$ is a suitable potential function, $\lambda$ and $\sigma$ are nonnegative constants and $W$ denotes a $d$-dimensional standard Brownian motion. Clearly, $\sigma \, \mathrm{d}W_t = \sigma \circ \mathrm{d}W_t$. As mentioned in the introduction we imaginize this equation as some kind of general kinetic model and still call it the Langevin equation. This process is placed at time $t \geq 0$ at the position $x_t \in \mathbb{R}^d$ and has an attached velocity vector $v_t \in \mathbb{R}^d$, thus its state space is interpreted as the tangent bundle of $\mathbb{R}^{d}$. In the following we derive a natural manifold-valued analogue of the Langevin equation. For the same reasons as before this process must live on the tangent bundle $\mathbb{T}\mathbb{M}$. Here the manifold $\mathbb{M} \subset \mathbb{R}^N$ can be chosen in its general form as defined in Section \ref{Section_Setup}. The final equation is then called the \textit{geometric Langevin equation}.

We explicitly remark again that our geometric Langevin equation has already been detected in its form before, see \cite[Eq.~(3.3)]{LRS12}. Nevertheless, the strategies for deriving the equation are different. The original motivation for us to derive a geometric Langevin equation arised during the development of a three-dimensional smooth fiber lay-down model, see the motivation from the introduction of this article. The strategy to obtain the desired geometric Langevin equation is summarized and explained in the introduction. We make use of the latter  again in Section \ref{Section_Applications}.

Furthermore, we remark that the transformation strategy below is similiar to the one in \cite{GKMS12}. In the latter article we have given a mathematical precise geometric derivation of the (original) three-dimensional fiber lay-down model from \cite{KMW12}. Consequently, original ideas for such a strategy can be found in \cite{KMW12} and have been the starting point for all our differentialgeometric considerations. Now let us start first with the formulation of the Langevin equation on $\mathbb{R} / {2 \pi \mathbb{Z}} \times \mathbb{R}$. We recall that all necessary statements concerning Stratonovich SDEs on manifolds are summarized in Section \ref{Appendix_Strat_manifolds}, see \cite{Hsu02} for details.

\subsection{The Langevin equation on \texorpdfstring{$\mathbb{R} / {2 \pi \mathbb{Z}} \times \mathbb{R}$}{}} \label{section_langevin_torus} Now we formulate the Langevin equation \eqref{Langevin_equation_Rn} in its most natural way on the tangent bundle of $\mathbb{R} / {2 \pi \mathbb{Z}}$, i.e., on $\mathbb{R} / {2 \pi \mathbb{Z}} \times \mathbb{R}$. Any point $p \in \mathbb{R} / {2 \pi \mathbb{Z}} \times \mathbb{R}$ is written in the form $p=(\alpha,v)$. We define the abstract Stratonovich Langevin equation on $\mathbb{R} / {2 \pi \mathbb{Z}} \times \mathbb{R}$ as
\begin{align} \label{Langevin_Torus}
\mathrm{d} X_t = \mathcal{A}_0(X_t)\, \mathrm{dt} + \mathcal{A}_1(X_t) \circ \mathrm{d}W_t
\end{align}
where $W=\{W_t\}_{t \geq 0}$ is a standard one-dimensional Brownian motion and the vector fields $\mathcal{A}_0$, $\mathcal{A}_1$ are defined by
\begin{align*}
\mathcal{A}_0 = v \, \frac{\partial}{\partial \alpha}(p) - \lambda \, v \frac{\partial}{\partial v}(p) - \frac{\partial \Phi}{\partial \alpha}(\alpha) \, \frac{\partial}{\partial v}(p),~\mathcal{A}_1 = \sigma \, \frac{\partial}{\partial v}(p).
\end{align*}
Here $\Phi \in C^\infty(\mathbb{R} / {2 \pi \mathbb{Z}})$. Let us justify the latter equation. Consider the canonical projection 
\begin{align*}
P : \mathbb{R} \rightarrow \mathbb{R} / {2 \pi \mathbb{Z}},~x \mapsto \alpha = [x].
\end{align*}
Now take any solution $Y_t=(x_t,v_t)$, $t \geq 0$, to the Langevin equation \eqref{Langevin_equation_Rn} in $\mathbb{R}^2$ in which the potential $\Psi$ is chosen as $\Psi := \Phi \circ P$. Let $f \in C^\infty(\mathbb{R} / {2 \pi \mathbb{Z}} \times \mathbb{R})$ and let $g \in C^\infty(\mathbb{R}^2)$ be given by $g(x,v) := f(P(x),v)$. Define $X$ by $X_t=(\alpha_t,v_t)$, $\alpha_t = P(x_t)$, $t \geq 0$. We have 
\begin{align*}
g(Y_t)=f(X_t),~\frac{\partial g}{\partial x}(x,v) = \frac{\partial f}{\partial \alpha} (\alpha,v),~(x,v) \in \mathbb{R}^2,~\alpha=P(x).
\end{align*}
Thus together with the Stratonovich transformation rule (see e.g.~\cite{Hsu02}) we obtain
\begin{align*}
\mathrm{d} f(X_t) = \mathrm{d}  g(Y_t) &= v_t \, \frac{\partial g}{\partial x}(Y_t) \, \mathrm{dt} - \left( \lambda v_t  + \partial_x \Psi(x_t)  \right) \, \frac{\partial g}{\partial v}(Y_t) \, \mathrm{dt} + \sigma \frac{\partial g}{\partial v}(Y_t) \circ \mathrm{d}W_t \\
& = (\mathcal{A}_0 f)(X_t) \,\mathrm{dt} + (\mathcal{A}_1 f)(X_t) \circ \mathrm{d}W_t.
\end{align*}
In other words, $X$ solves equation \eqref{Langevin_Torus}. This justifies Equation \eqref{Langevin_Torus} on the tangent bundle of $\mathbb{R} / {2 \pi \mathbb{Z}}$ in a natural way.

\subsection{The Langevin equation on \texorpdfstring{$\mathbb{T}\mathbb{S}^1$}{} and \texorpdfstring{$\mathbb{T}\mathbb{S}^d$}{}} \label{section_langevin_sphere} Now $\mathbb{R} / {2 \pi \mathbb{Z}} \times \mathbb{R}$ is diffeomorph to $\mathbb{T}\mathbb{S}^1$, where $\mathbb{S}^1$ is the unit sphere embedded in $\mathbb{R}^2$. Consequently, the Langevin equation on $\mathbb{T}\mathbb{S}^1$ can be derived via calculating the equivalent SDE to  Equation \eqref{Langevin_Torus} on $\mathbb{T}\mathbb{S}^1$. But first note
\begin{align*}
\mathbb{T} \mathbb{S}^1 = \left\{ (\xi,\omega) \in \mathbb{R}^4~|~|\xi|^2=1,~\omega \cdot \xi = 0 \right\}
\end{align*}
Hence $\mathcal{V}$ is a smooth vector field on $\mathbb{T} \mathbb{S}^1$ if and only if
\begin{align} \label{Condition_TS1}
\begin{pmatrix} \xi \\ 0 \end{pmatrix} \cdot \mathcal{V}(\xi,\omega)  = 0,~ \begin{pmatrix} \omega \\ \xi \end{pmatrix} \cdot \mathcal{V}(\xi,\omega) = 0
\end{align}
holds for all $(\xi,\omega) \in \mathbb{T} \mathbb{S}^1$. The previously mentioned diffeomorphism from $\mathbb{R} / {2 \pi \mathbb{Z}} \times \mathbb{R}$ to $\mathbb{T} \mathbb{S}^1$ is now given as
\begin{align*}
(\alpha,v) \mapsto (\xi,\omega)= \begin{pmatrix} \varphi(\alpha) \\ ~v\,\varphi^\bot(\alpha) \end{pmatrix},~\varphi(\alpha) := \begin{pmatrix} \cos(\alpha) \\ \sin(\alpha) \end{pmatrix},~\varphi^\bot(\alpha) := \begin{pmatrix} -\sin(\alpha) \\ \cos(\alpha) \end{pmatrix},
\end{align*}
which is induced by the diffeomorphism $\varphi: \mathbb{R} / {2 \pi \mathbb{Z}} \rightarrow \mathbb{S}^1$. Analogously as in Section \ref{Section_Setup}, see Equation \eqref{eq_computing_push_forward}, we get
\begin{align*}
\widetilde{\frac{\partial}{\partial \alpha}}\, \equiv  \, \begin{pmatrix} \varphi^\bot \\ - v \, \varphi \end{pmatrix},~\widetilde{\frac{\partial}{\partial v}} \, \equiv  \, \begin{pmatrix} 0 \\  \varphi^\bot \end{pmatrix}.
\end{align*}
Here the tilde denotes the push-forward vector fields on $\mathbb{T} \mathbb{S}^1$. Define $\widetilde{\Phi} \in C^\infty(\mathbb{S}^1)$ via $\widetilde{\Phi} := \Phi \circ \varphi^{-1}$. Then 
\begin{align*}
\frac{\partial \Phi}{\partial \alpha} (\alpha) = \varphi^\bot(\alpha) \cdot \nabla_\xi \,\widetilde{\Phi} (\xi) = \xi^\bot \cdot \nabla_\xi \,\widetilde{\Phi} (\xi),~\xi^\bot = \begin{pmatrix} - \xi_2 \\ \xi_1 \end{pmatrix}.
\end{align*}
Thus $\frac{\partial \Phi}{\partial \alpha} \varphi^\bot = \text{grad}_{\mathbb{S}^1} \widetilde{\Phi} (\xi)$ since $T_\xi \mathbb{S}^1=\text{span}\{ \xi^\bot \}$. Furthermore,  $v^2 \varphi = |\omega|^2 \xi$. Thus
\begin{align*}
\widetilde{\mathcal{A}_0}\, \equiv  \,\begin{pmatrix}  \omega \\ - |\omega|^2 \xi - \lambda \omega - \text{grad}_{\mathbb{S}^1} \widetilde{\Phi} (\xi)\end{pmatrix},~\widetilde{\mathcal{A}_1}\, \equiv  \, \sigma \begin{pmatrix}  0 \\ \xi^\bot \end{pmatrix}.
\end{align*}
$\mathcal{A}_0$ and $\mathcal{A}_1$ are defined in the previous section. So the equivalent Stratonovich SDE to \eqref{Langevin_Torus} for the Langevin equation on $\mathbb{T} \mathbb{S}^1$ reads
\begin{align} \label{Langevin_S1}
\mathrm{d} X_t = \widetilde{\mathcal{A}_0}(X_t)\, \mathrm{dt} + \widetilde{\mathcal{A}_1}(X_t) \circ \mathrm{d}W_t
\end{align}
where $W$ is a standard one-dimensional Brownian motion. Observe that $\widetilde{\mathcal{A}_0},~\widetilde{\mathcal{A}_1}$ really satisfy the conditions in \eqref{Condition_TS1}. In order to obtain a natural translation of the Langevin equation to $\mathbb{T}\mathbb{S}^d$, $\mathbb{S}^d$ the unit sphere in $\mathbb{R}^{d+1}$, we slightly change the stochastic part in \eqref{Langevin_S1}. Therefore, we replace $\widetilde{\mathcal{A}_1} \circ dW_t$ through 
\begin{align*}
(I-\xi \otimes \xi) \circ \mathrm{d}W_t=\sum_{j=1}^2 \begin{pmatrix} 0 \\ (I - \xi \otimes \xi )e_j \end{pmatrix} \circ \mathrm{d} W_t^{(j)}
\end{align*}
where $e_j$, $j=1,2$, denotes the $j$-th unit vector in $\mathbb{R}^2$ and $W$ now denotes a standard two-dimensional Brownian motion. The resulting Stratonovich SDE on $\mathbb{T}\mathbb{S}^1$ then differs from \eqref{Langevin_S1}, but it is easy to see that its generator is exactly the same as the one from \eqref{Langevin_S1}. In particular, any solution to the new modified equation on $\mathbb{T}\mathbb{S}^1$ then coincides weakly with any solution to \eqref{Langevin_S1}. Now we may write the resulting equation, directly translated to $\mathbb{T}\mathbb{S}^d$, simply as some Stratonovich SDE in $\mathbb{R}^{2(d+1)}$ in the form
\begin{align} \label{Langevin_equation_sphere}
&\mathrm{d} \xi_t = \omega_t \, \mathrm{dt} \\
&\mathrm{d} \omega_t =  - \lambda \,\omega_t \,\mathrm{dt} - |\omega_t|^2 \xi_t \,\mathrm{dt} - \text{grad}_{\mathbb{S}^d} \, \Phi (\xi_t) \,\mathrm{dt} + \sigma \,\Pi_{\mathbb{S}^d}[\xi_t] \circ \mathrm{d}W_t. \nonumber
\end{align}
The solution to the latter equation stays on $\mathbb{T} \mathbb{S}^d$ provided that the initial value lies on the manifold. Here $W$ is a standard $(d+1)$-dimensional Brownian motion and recall $\Pi_{\mathbb{S}^d} [\xi] = I - \xi \otimes \xi$, $\xi \in \mathbb{S}^d$. By abuse of notation, the potential is written again without tilde. 

\subsection{The geometric Langevin equation on \texorpdfstring{$\mathbb{T}\mathbb{M}$}{}} \label{section_geometric_langevin} Let $\mathbb{M} \subset \mathbb{R}^{N}$ be as in Section \ref{Section_Setup}. Having in mind the Langevin equation on $\mathbb{T}\mathbb{S}^d$ from \eqref{Langevin_equation_sphere} we propose a natural ansatz for the \textit{geometric Langevin equation} with state space $\mathbb{T} \mathbb{M}$ as the following abstract Stratonovich SDE 
\begin{align} \label{Langevin_equation_geometric} 
\mathrm{d}X_t = \mathcal{N}_1 (X_t) \, \mathrm{dt} + \mathcal{N}_2 (X_t) \, \mathrm{dt} + \mathcal{N}_3 (X_t) \, \mathrm{dt} + \sum_{j=1}^N \mathcal{M}_j(X_t) \circ \mathrm{d}W_t^{(j)}
\end{align}
where $W$ is a standard $N$-dimensional Brownian motion and the vector fields $\mathcal{N}_i$ and $\mathcal{M}_j$ are defined on $\mathbb{T}\mathbb{M}$ as
\begin{align} \label{vector_fields_Langevin_geometric}
\mathcal{N}_1= \begin{pmatrix} w \\ - F(\xi,\omega) \end{pmatrix},~\mathcal{N}_2  = \begin{pmatrix} 0 \\ - \lambda \,\omega \end{pmatrix},~\mathcal{N}_3  = \begin{pmatrix} 0 \\ -  \text{grad}_{\mathbb{M}} \, \Phi (\xi) \end{pmatrix},~\mathcal{M}_j  = \sigma \begin{pmatrix} 0 \\  \Pi_{\mathbb{M}}[\xi]e_j \end{pmatrix}.
\end{align}
Here $e_j$, $j=1,\ldots,N$, is the $j$-th unit vector in $\mathbb{R}^{N}$, $\Phi \in C^\infty(\mathbb{M})$, and the smooth forcing term $F : \mathbb{T}\mathbb{M} \rightarrow \mathbb{R}^{N}$ has to be chosen such that the vector vield $\mathcal{N}_1$ is really tangential to $\mathbb{T} \mathbb{M}$. By condition \eqref{tangent_condition_tangentbundle} this means
\begin{align*}
\begin{pmatrix} \,\omega \cdot \nabla^2 f_1 (\xi) \,\omega   \\ \vdots \\ \,\omega \cdot \nabla^2 f_k (\xi) \,\omega    \end{pmatrix} =J_\xi \,F(\xi,\omega),~(\xi,\omega) \in \mathbb{T}\mathbb{M}.
\end{align*}
So the tangential condition \eqref{tangent_condition_tangentbundle} holds by setting
\begin{align} \label{Definition_forcing_term_F}
F(\xi,\omega) := J_\xi^T \,(J_\xi \,J_\xi^T)^{-1} \begin{pmatrix} \,\omega \cdot \nabla^2 f_1 (\xi) \,\omega   \\ \vdots \\ \,\omega \cdot \nabla^2 f_k (\xi) \,\omega    \end{pmatrix},~(\xi,\omega) \in \mathbb{T}\mathbb{M}.
\end{align}
Note that $\mathcal{N}_2$, $\mathcal{N}_3$ and all $\mathcal{M}_j$ indeed satisfy the tangential condition \eqref{tangent_condition_tangentbundle}. As described in Section \ref{Appendix_Strat_manifolds}, the solution to the abstract geometric Langevin Equation \eqref{Langevin_equation_geometric} is obtained by solving the following Stratonovich SDE in $\mathbb{R}^{2N}$
\begin{align} \label{Langevin_equation_geometric_abstract}
&\mathrm{d} \xi_t = \omega_t \, \mathrm{dt} \\
&\mathrm{d} \omega_t =  - \lambda \,\omega_t \,\mathrm{dt} - F(\xi_t,\omega_t) \,\mathrm{dt} - \text{grad}_{\mathbb{M}} \, \Phi (\xi_t) \,\mathrm{dt} + \sigma \,\Pi_{\mathbb{M}}[\xi_t] \circ \mathrm{d}W_t \nonumber
\end{align}
whose solution stays on $\mathbb{T}\mathbb{M}$ provided that the initial value lies on $\mathbb{T}\mathbb{M}$. Of course, the occuring vector fields from $\mathbb{R}^{2N}$ in \eqref{Langevin_equation_geometric_abstract} have to be understood as arbitrary smooth extensions of the vector fields from $\mathbb{T}\mathbb{M}$ in  \eqref{Langevin_equation_geometric}. Thus our modeling of the geometric Langevin equation is complete. At this point we mention again \cite{LRS12} or \cite[Sec.~3.3]{LRS10} for an alternative, different derivation of the geometric Langevin process in the context of constrained Langevin dynamics.

\begin{Rm}  Furthermore, note that in case $\mathbb{M}=\mathbb{S}^d$ or $\mathbb{M}=\mathbb{R}^d$ we get back \eqref{Langevin_equation_sphere} or \eqref{Langevin_equation_Rn} respectively. In the second case recall the identification of $\mathbb{R}^d$ made in \eqref{identification_Rd}. Of course, we neglect afterwards the redundant $(d+1)$-component of $\xi$ and $\omega$ in the resulting Langevin equation in the $\mathbb{R}^d$-case. Finally, note that the embedded SDE \eqref{Langevin_equation_geometric} writes as It\^{o}-SDE in $\mathbb{R}^{2N}$ in exactly the same form. 
\end{Rm}

\section{The Langevin generator} \label{Langevin_Generator}

Let us cite the mathematical articles dealing with generalized geometric Langevin processes. Besides the already mentioned article \cite{LRS12} (or \cite[Sec.~3.3]{LRS10} respectively), consider the book of Gliklikh, see \cite{Gli97}, for the construction and discussion of Langevin type equations with external forcing field arising in geometric mechanics. In the latter, the It\^{o} SDE approach (involving It\^{o} bundles) on general Riemannian manifolds is used. Furthermore, consider \cite{Jor78} where the Langevin process on the tangent bundle of smooth manifolds is constructed via the Gangolli-McKean injection scheme. Therein, J{\o}rgensen provides a general expression in local coordinates for the generator associated with the Langevin process. A forcing field (or external potential) is not included there. And as described in the introduction, Soloveitchik, see \cite{Sol95}, also considers a construction approach via defining a suitable generalized Langevin (or called Ornstein-Uhlenbeck therein) generator $L$. This generator extends the one from J{\o}rgensen and includes the additional potential term. Finally, consider the book of Kolokoltsov, see \cite{Kol00}, in which the author constructs a curvilinear Ornstein-Uhlenbeck process on the cotangent bundle of some compact $d$-dimensional manifold via introducing a suitable operator extending the one from Soloveitchik, see \cite[Ch.~4,~Sec.~1]{Kol00}. Some of the above mentioned articles are seemingly not known to each other, see the references therein.

In the following, we calculate the generator associated to our previously defined geometric Langevin process and show that it coincides with the one introduced by Soloveitchik. Such a local representation is not contained in \cite{LRS12}, or \cite{LRS10} respectively. This connects both approaches and moreover, yields a desirable way to formulate \eqref{Langevin_equation_geometric} in local coordinate form, see Remark \ref{Rm_simulating_geometric_Langevin}.

\begin{Pp} \label{Pp_Langevin_generator}
Let $\mathbb{M}$ be a $d$-dimensional regular submanifold of $\mathbb{R}^{N}$ as in Section \ref{df_and_not}, assume that $\Phi \in C^\infty(\mathbb{M})$ and let $(U,q)$ be a local coordinate chart of $\mathbb{M}$. The Kolmogorov operator $L:C^\infty(\mathbb{T}\mathbb{M}) \rightarrow C^\infty(\mathbb{T}\mathbb{M})$ associated to the geometric Langevin equation \eqref{Langevin_equation_geometric} on $\mathbb{T}\mathbb{M}$ is given in the corresponding coordinate chart $(TU,Tq)$, $Tq=(q^1,\ldots,q^d,v^1,\ldots,v^d)$, on $\mathbb{T}\mathbb{M}$ as
\begin{align*}
L=\sum \left( v^j \frac{\partial}{\partial q^j} -  \Gamma^j_{nm}  v^n v^m \frac{\partial}{\partial v^j} - g^{ij}\frac{\partial \Phi}{\partial q^i} \frac{\partial}{\partial v^j} - \lambda \, v^{j} \frac{\partial}{\partial v^j} \right) + \frac{1}{2} \sigma^2 \sum_{i,j} g^{ij} \frac{\partial^2}{\partial v^i \partial v^j}.
\end{align*}
\end{Pp}

\begin{proof}

\textit{Step 1:} First remember the local diffeomorphism $(Tq)^{-1}: q(U) \times \mathbb{R}^d \rightarrow TU$ mapping $(s,\kappa)$ to $(\tau(s), \sum_{j} \kappa_j \partial_j \tau(s))$, see \eqref{eq_diffeomorphism_chart}. Here $\partial_j \tau = \frac{\partial \tau}{\partial s_j}$. Under $(Tq)^{-1}$ the smooth vector fields $\frac{\partial}{\partial{s_i}},~\frac{\partial}{\partial{\kappa_j}}$, on $q(U) \times \mathbb{R}^d$ correspond to the smooth vector fields
\begin{align} \label{eq_corresponding_vectro fields}
\widetilde{\frac{\partial}{\partial{s_i}}} =\frac{\partial}{\partial{q^i}}\equiv \begin{pmatrix} \partial_i \tau  \\ \sum_{j=1}^d \kappa_j \partial_{ij} \tau \end{pmatrix},~\widetilde{\frac{\partial}{\partial{\kappa_j}}}=\frac{\partial}{\partial{v^j}} \equiv \begin{pmatrix} 0 \\ \partial_{j} \tau \end{pmatrix}
\end{align}
from $TU$. This follows in the same way as in \eqref{eq_computing_push_forward}.

\textit{Step 2:} The Langevin generator $L$ is given on $C^\infty(\mathbb{T}\mathbb{M})$ as
\begin{align*}
L=\mathcal{N}_1 + \mathcal{N}_2 + \mathcal{N}_3 + \frac{1}{2} \sum_{j} \mathcal{M}_j^2,
\end{align*}
see Equation \eqref{vector_fields_Langevin_geometric}. We write each $(\xi,\omega) \in TU$ again in the form $(\tau(s), \sum_{j} \kappa_j \partial_j \tau(s))$ with $(s,\kappa) \in q(U) \times \mathbb{R}^d$ as introduced in Section \ref{df_and_not}. Let us calculate the vector fields $\mathcal{N}_j$ and $\mathcal{M}_i$ in the local coordinate system $(TU,Tq)$. We start with $\mathcal{N}_1$. By the previous, there exists uniquely determined $a_j(\xi,\omega),b_j(\xi,\omega) \in \mathbb{R}$ such that
\begin{align*}
\sum_{j} a_j (\xi,\omega) \, \frac{\partial}{\partial{q^j}}(\xi,\omega) + \sum_{j} b_j (\xi,\omega) \, \frac{\partial}{\partial{v^j}}(\xi,\omega) \, \equiv\, \mathcal{N}_1(\xi,\omega) = \begin{pmatrix} \omega \\ - F(\xi,\omega) \end{pmatrix}
\end{align*} 
For notational convenience we omit the argument $(\xi,\omega)$ in the following. By \eqref{eq_corresponding_vectro fields} and the uniqueness of the representation for $\omega$ we conclude that $a_j= \kappa_j$ must hold for all $j=1,\ldots,d$. Now by taking the scalar product with respect to some $\begin{pmatrix} 0 \\ \partial_i \tau \end{pmatrix}$ on both sides of the latter equation and using again \eqref{eq_corresponding_vectro fields} we get
\begin{align*}
\sum_{n,m} \kappa_n \,\kappa_m \,\partial_i \tau \cdot \partial_{nm} \tau +  \sum_{n} b_n \, \partial_n \tau \cdot \partial_i \tau =0.
\end{align*}
Here we have used the fact $F \cdot \partial_i \tau=0$ which follows since $J_\xi \,v=0$ holds for all $v \in T_\xi \mathbb{M}$ by definition. By using additionally the relation for the Christoffel symbols, see formula \eqref{eq_relation_Christoffel symbols}, we get
\begin{align*}
b_j = \sum_{n} b_n \delta_{nj} = \sum_{n,i}  b_n \, g_{ni}\, g^{ij} = - \sum_{n,m} \kappa_n \,\kappa_m \,\sum_{i} g^{ij} \,\partial_i \tau \cdot \partial_{nm} \tau = - \sum_{n,m} \Gamma_{nm}^{j} \,\kappa_n \,\kappa_m
\end{align*}
for all $j=1,\ldots,d$. With $v^j(\xi,\omega)=\kappa_j$ this shows
\begin{align*}
{\mathcal{N}_1} \equiv \sum_{j} v^j \, \frac{\partial}{\partial{q^j}} - \sum_{j,n,m} \Gamma_{nm}^{j}(q) \,  v^n \,v^m \, \frac{\partial}{\partial{v^j}}
\end{align*}
on $TU$. And due to $\omega \cdot \nabla_\omega = \sum_{j} \kappa_j \, \partial_j \tau(s) \cdot \nabla_\omega = \sum_{j} \kappa_j \, \frac{\partial}{\partial{v^j}}$ we conclude
\begin{align*}
{\mathcal{N}_2} \equiv - \lambda \sum_{j} v^j \, \frac{\partial}{\partial{v^j}} 
\end{align*}
on $TU$. Next we write 
\begin{align*}
\text{grad}_{\mathbb{M}} \Phi (\xi) = \left(I - J_\xi^{\,T} \left(J_\xi \, J_\xi^{\,T}\right)^{-1} J_\xi \right)\,\nabla_\xi \Phi(\xi) = \sum_{j} c_j(\xi) \, \partial_j\tau
\end{align*}
for some uniquely determined $c_j(\xi) \in \mathbb{R}$. Hence it holds $\sum_{n} c_n \, g_{in} = \nabla_{\xi} \Phi \cdot \partial_i \tau$ for each $i=1,\ldots,d,$ since $J_\xi \,\partial_i \tau =0$. Consequently,
\begin{align*}
c_j = \sum_{n} c_n \,\delta_{nj} = \sum_{n,i} c_n \,g_{ni} \,g^{ij} =  \sum_{i} g^{ij}\,{ \nabla_{\xi} \Phi  \cdot \partial_i\tau}  = \sum_{i} g^{ij} \, \frac{\partial \Phi }{\partial{q^i}} 
\end{align*}
for all $j=1,\ldots,d$. Altogether, we obtain 
\begin{align*}
{\mathcal{A}_3}\equiv - \sum_{i,j}  g^{ij} \, \frac{\partial \Phi }{\partial{q^i}} \, \frac{\partial}{\partial{v^j}}.
\end{align*}
on $TU$. Now we choose some $r \in \{ 1,\ldots,N\}$. We write
\begin{align*}
\Pi_\mathbb{M}[\xi](e_r)=\left(I - J_\xi^{\,T} \left(J_\xi \, J_\xi^{\,T}\right)^{-1} J_\xi \right) e_r = \sum_{i} d^{\,r}_i \, \partial_i \tau
\end{align*}
for some uniquely determined $d^{\,r}_i=d^{\,r}_i(\xi,\omega) \in \mathbb{R}$, $i=1,\ldots,d$. With a similiar calculation as above we obtain $d^{\,r}_i = \sum_{j} g^{ij} \, { e_r \cdot \partial_j \tau}$ for all $i=1,\ldots,d$. Thus 
\begin{align*}
\mathcal{M}_r \equiv \sigma  \left(I - J_\xi^{\,T} \left(J_\xi \, J_\xi^{\,T}\right)^{-1} J_\xi \right) e_r \cdot \nabla_\omega = \sigma  \, \sum_{i} d^{\,r}_i \, \partial_i\tau \cdot \nabla_\omega = \sigma \, \sum_{i,j} g^{ij} \, { e_r \cdot \partial_j\tau} \, \frac{\partial}{\partial{v^i}}.
\end{align*}
Consequently, we get
\begin{align*}
\mathcal{M}_r^2 \equiv \sigma^2 \sum_{i,j,n,m} g^{ij} g^{nm} \, \big( e_r,   \partial_j \tau \big)_{\text{euc}}  \big( e_r , \partial_m\tau \big)_{\text{euc}} \, \frac{\partial^2}{\partial{v^i} \partial{v^{n}}}.
\end{align*}
This finally implies
\begin{align*}
\sum_{r} \mathcal{M}_r^2 \equiv \sigma^2 \sum_{i,j,n,m} g^{ij} g^{nm} g_{mj} \frac{\partial^2}{\partial{v^i} \partial{v^{n}}} = \sigma^2 \sum_{i,j} g^{ij} \frac{\partial^2}{\partial{v^i} \partial{v^{j}}}
\end{align*}
and the desired formula is proven.
\end{proof}

\begin{Rm} \label{Rm_simulating_geometric_Langevin}
In order to simulate the geometric Langevin process numerically, we may of course consider SDE \eqref{Langevin_equation_geometric} in $\mathbb{R}^{2N}$ directly. Nevertheless, this requires a numerical algorithm which stays on $\mathbb{T}\mathbb{M}$. So let us describe an alternative simulation method: Note that for a given coordinate system $Tq$ of $\mathbb{T} \mathbb{M}$ as above, we may directly write down a stochastic differential equation in the obvious way for each such local coordinates having the (local) Langevin operator $L$ from Proposition \ref{Pp_Langevin_generator} as associated Kolmogorov operator. In other words, this so called local SDE gives a (local) $L$-diffusion, thus yields a convenient way to simulate the (weakly unique) geometric Langevin diffusion process obtained from \eqref{Langevin_equation_geometric}. We illustrate this method in the following applications.
\end{Rm}

\section{The Spherical Langevin process} \label{Spherical_Langevin}

In practice there are many kinetic models availabe dealing with manifold-valued Brownian motions, especially spherical ones. Nevertheless, for many applications this gives to rough paths and it is reasonable to search for smoother versions of the classical models. Therefore, let us consider the following general point of view: The geometric Langevin process on $\mathbb{T}\mathbb{M}$, or more precisely its $\xi$-coordinates, may be seen for general potential $\Phi$ as some kind of a smooth analogon to the Ornstein-Uhlenbeck process on $\mathbb{M}$. The latter is given by the Stratonovch SDE on $\mathbb{M}$ of the form
\begin{align} \label{Ornstein_Uhlenbeck_sphere}
\mathrm{d} \xi_t = -\frac{\sigma^2}{2} \text{grad}_{\mathbb{M}} \Phi (\xi_t) \, \mathrm{dt} + \sigma \, \Pi_{\mathbb{M}}[\xi_t] \circ \mathrm{d}W_t.
\end{align}
Here $W$ denotes a standard $N$-dimensional Brownian motion. In particular, for $\Phi=0$, Equation \eqref{Ornstein_Uhlenbeck_sphere} gives a Brownian motion on $\mathbb{M}$ (with diffusion constant $\sigma$), for details see e.g.~\cite[Prop.~3.2.6.]{Hsu02} and \cite[Theo.~3.1.4.]{Hsu02}. In this situation, the corresponding Langevin process on $\mathbb{T}\mathbb{M}$ with $\Phi=0$ can be used as the smooth analogon to the Brownian motion on $\mathbb{M}$ and yields the development of new, extended kinetic models, see for example the smooth fiber lay-down model developed in Section \ref{Section_Applications}. In the latter section we are dealing with examples involving the spherical situation, i.e., $\mathbb{M}=\mathbb{S}^2$. In view of these applications and in order to illustrate the strategy described in Remark \ref{Rm_simulating_geometric_Langevin}, we shall seperately discuss now the geometric Langevin process living on $\mathbb{T}\mathbb{S}^2$.

Therefore, consider the special coordinate charts $(U,q)$ of $\mathbb{S}^2$ given as the inverse of $\tau:(a,a+2\pi) \times (0,\pi) \rightarrow \mathbb{R}$, $a \in \mathbb{R}$, where
\begin{align*}
\tau(\theta_1,\theta_2):= \left(\cos \theta_1 \sin \theta_2, \sin \theta_1 \sin\theta_2 , \cos\theta_2 \right)^T,~\theta_1 \in (a,a+2\pi),~\theta_2 \in (0,\pi).
\end{align*}
By using formula \eqref{eq_relation_Christoffel symbols} one easily calculates
\begin{align*}
&g_{11}=\sin^2\theta_2,~g_{22}=1,~g^{11}=\frac{1}{\sin^2\theta_2},~g^{22}=1\\
&\Gamma^1_{12}=\Gamma^1_{21}=\cot\theta_2,~\Gamma_{11}^2=-\sin\theta_2\cos\theta_2,
\end{align*}
and $0$ else. So $q^1=\theta_1,~q^2=\theta_2$ and  $v^1=\kappa_1,~v^2=\kappa_2$. Thus a local (It\^{o} or Stratonovich) SDE simulating the local $L$-diffusion constructed from \eqref{Langevin_equation_geometric} reads as follows (for abuse of notation we omit the time index $t$):
\begin{align} \label{Langevin_sphere_local}
&\mathrm{d} \theta_1 = \kappa_{1}\, \mathrm{dt},~\mathrm{d} \theta_2 = \kappa_{2}\, \mathrm{dt}  \\
&\mathrm{d} \kappa_{1} = -\lambda \kappa_{1} \, \mathrm{dt} - 2 \cot\theta_2\kappa_{1}\kappa_{2}\, \mathrm{dt} - \frac{1}{\sin^2\theta_2}\frac{\partial \Phi}{\partial \theta_1}(\theta_1,\theta_2)  \, \mathrm{dt} + \sigma \frac{1}{\sin\theta_2} \,\mathrm{d}W^{(1)}_t  \nonumber \\
&\mathrm{d} \kappa_{2} = -\lambda \kappa_{2} \, \mathrm{dt} +  \sin\theta_2\cos\theta_2 \kappa_{1}^2 \, \mathrm{dt}- \frac{\partial \Phi}{\partial \theta_2}(\theta_1,\theta_2)  \, \mathrm{dt} + \sigma  \,\mathrm{d}W^{(2)}_t. \nonumber
\end{align}

\begin{figure}[tbp] 
\subfigure{\includegraphics[scale=0.42]{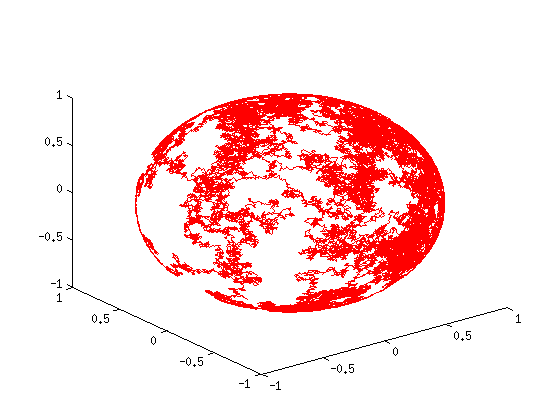}}
\subfigure{\includegraphics[scale=0.42]{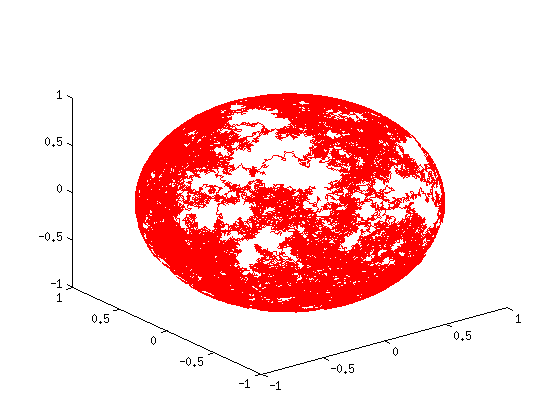}}
\caption{Spherical Brownian motion} \label{figure_spherical_BM}
\end{figure}

Here $W$ denotes a standard two-dimensional Brownian motion and due to the $2\pi$-periodicity of $\tau$ in the variable $\theta_1$, the state space of the previous SDE may be considered as $\mathbb{R} / {2 \pi \mathbb{Z}}  \times (0,\pi) \times \mathbb{R}^2$. In this case the SDE shall more precisely be understood as an abstract Stratonovich equation similiar as in Section \ref{section_langevin_torus}. The latter can be simulated via considering Equation \eqref{Langevin_sphere_local} in $\mathbb{R} \times (0,\pi) \times \mathbb{R}^2$ directly.

But first, we do a transformation of \eqref{Langevin_sphere_local} on $\mathbb{R} / {2 \pi \mathbb{Z}}  \times (0,\pi) \times \mathbb{R}^2$ via
\begin{align*}
(\theta_1,\theta_2,\kappa_1,\kappa_2) \mapsto (\theta_1, \theta_2, \nu_1, \nu_2) :=(\theta_1, \theta_2, \kappa_1 \sin\theta_2, \kappa_2).
\end{align*}
With these new local coordinates,  an easy calculation shows that \eqref{Langevin_sphere_local} can equivalently be formulated on $\mathbb{R} / {2 \pi \mathbb{Z}}  \times (0,\pi) \times \mathbb{R}^2$ as
\begin{align} \label{Langevin_sphere_local_2}
&\mathrm{d} \theta_1 = \frac{\nu_1}{\sin\theta_2} \, \mathrm{dt},~\mathrm{d} \theta_2 = \nu_2 \, \mathrm{dt} \\
&\mathrm{d} \nu_1 = -\lambda  \nu_1 \, \mathrm{dt} -  \cot\theta_2\nu_1 \nu_2 \, \mathrm{dt}- \frac{1}{\sin\theta_2}\frac{\partial \Phi}{\partial \theta_1}(\theta_1,\theta_2) \, \mathrm{dt} + \sigma  \,\mathrm{d}W^{(1)}_t  \nonumber \\
&\mathrm{d} \nu_2 = -\lambda  \nu_2 \, \mathrm{dt}   +  \cot\theta_2 \nu_1^2 \,\mathrm{dt}- \frac{\partial \Phi}{\partial \theta_2}(\theta_1,\theta_2)\, \mathrm{dt} + \sigma  \,\mathrm{d}W^{(2)}_t. \nonumber
\end{align}
These new local coordinates exactly are obtained by the (more natural) parametrization of $\mathbb{T}\mathbb{S}^2$ given through
\begin{align*}
\mathbb{R} / {2 \pi \mathbb{Z}}  \times (0,\pi) \times \mathbb{R}^2 \ni (\theta_1,\theta_2,\nu_1,\nu_2) \mapsto (\tau, \nu_1 \,n_1 + \nu_2 \,n_2)
\end{align*}
where $n_1=\frac{1}{\left|\partial_{\theta_1} \tau \right|} \partial_{\theta_1} \tau$ and $n_2 = \partial_{\theta_2} \tau$ are the spherical unit vectors. Compare this with the general parametrization from \eqref{eq_diffeomorphism_chart}. Morever, a stationary solution to the (formal) Fokker-Planck equation associated with the local SDE \eqref{Langevin_sphere_local_2} now takes the desired well-known form $e^{-\Phi}e^{- \frac{\lambda}{\sigma^2} \nu_1^2}e^{- \frac{\lambda}{\sigma^2} \nu_2^2}\sin\theta_2$. 

\begin{figure}[tbp] 
\subfigure{\includegraphics[scale=0.42]{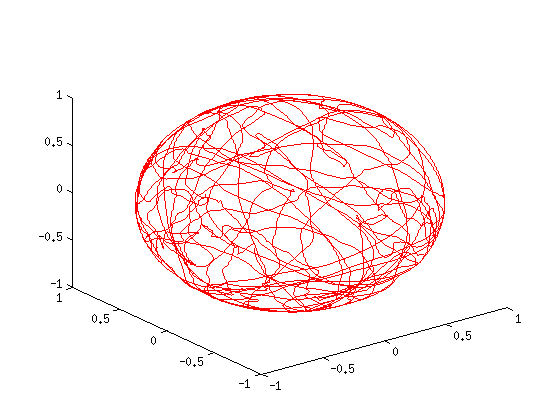}}
\subfigure{\includegraphics[scale=0.42]{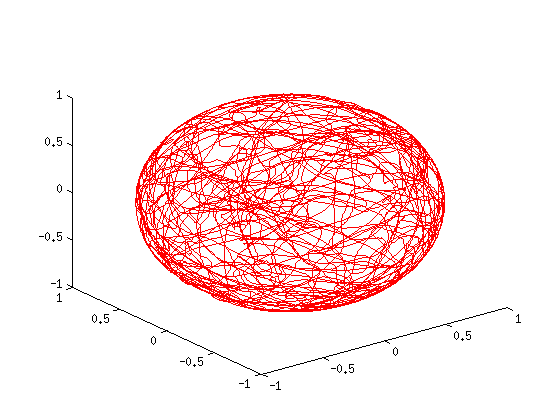}}
\caption{Spherical Langevin process} \label{figure_spherical_LP}
\end{figure}

A local SDE for the spherical Brownian motion with diffusion constant $\sigma$ is obtained via an SDE on $\mathbb{R} / {2 \pi \mathbb{Z}}  \times (0,\pi)$ of the form
\begin{align*}
&\mathrm{d} \theta_1 = \frac{\sigma}{\sin\theta_2} \, \mathrm{d}W_t^{(1)},~\mathrm{d} \theta_2 = \frac{\sigma^2}{2} \cot\theta_2\,\mathrm{dt} + \sigma \,\mathrm{d}W_t^{(2)}.
\end{align*}
This is due to the fact that the Kolmogorov operator associated with the latter equation coincides with the generator $\frac{\sigma^2}{2} \Delta_{\mathbb{S}^2}$ from \eqref{Ornstein_Uhlenbeck_sphere} (for $\Phi=0$) restricted to $\mathbb{R} / {2 \pi \mathbb{Z}} \times (0,\pi)$, see again \cite[Theo.~3.1.4]{Hsu02}. Figure \ref{figure_spherical_BM} shows the paths of spherical Brownian motion with $\sigma=1$. Compare the rough paths with the smooth ones from Figure \ref{figure_spherical_LP} where the $\xi$-coordinates of the spherical Langevin process are plotted for $\Phi=0$, $\lambda=1$ and $\sigma=1$. All simulations, now and in the following, are always obtained with a simple Euler-Maruyama scheme.

\section{A new smooth fiber lay-down model involving the spherical Langevin process and further applications to biology} \label{Section_Applications}

In this section we present a new smooth three-dimensional fiber lay-down model. It is fascinating to note that the stochastic part of this model is given by the spherical Langevin process living on $\mathbb{T}\mathbb{S}^2$ introduced previously. And moreover, as decribed in the introduction, we even present possible applications in biology to self-propelled interacting systems describing the collective behaviour of swarms of animals.

\subsection{A new smooth fiber lay-down model} \label{section_smooth_fiber_lay_down}

First remember the basic fiber lay-down Stratonovich  SDE \eqref{SDE_Basic_Fiber_Lay_Down} with state space $\mathbb{R}^d \times \mathbb{S}^{d-1}$. For convenience, we stay general and let $d \in \mathbb{N}$, $d \geq 2$. The aim is to develop a fiber lay-down model having smoother trajectories than the original one, see the motivation in Section \ref{Section_Geometry_of_Fiber_Lay_Down}. Thus basically, we have to replace the spherical Brownian motion (with diffusion constant $\sigma$) therein through the spherical Langevin process with potential equal to $0$ (and with diffusion constant $\sigma$) from \eqref{Langevin_equation_sphere}. Nevertheless, the latter lives on $\mathbb{T} \mathbb{S}^{d-1}$. Consequently, the state space of the smooth fiber lay-down model must be equal to $\mathbb{R}^d \times \mathbb{T} \mathbb{S}^{d-1}$. In notation
\begin{align*}
\mathbb{R}^d \times \mathbb{T} \mathbb{S}^{d-1} = \{ (\xi,\omega,\mu) \in \mathbb{R}^{d} \times \mathbb{R}^d \times \mathbb{R}^d~|~|\omega|^2=1,~\omega \cdot \mu =0 \}.
\end{align*} 
Thus the vector field $\begin{pmatrix} \omega \\  -(I-\omega \otimes \omega) \nabla \Phi(\xi) \end{pmatrix}$, $(\xi,\omega) \in \mathbb{R}^d \times \mathbb{S}^{d-1}$, with $\Phi \in C^\infty(\mathbb{R}^d)$ only depending on $\xi$, from the basic model has to be extended in the most natural way to some vector field $\mathcal{A}$ on $\mathbb{R}^d \times \mathbb{T} \mathbb{S}^{d-1}$. We propose the ansatz
\begin{align*}
\mathcal{A}(\xi,\omega,\mu):=\begin{pmatrix} \omega \\ -(I-\omega \otimes \omega) \nabla \Phi(\xi) \\ G(\xi,\omega,\mu) \end{pmatrix},~(\xi,\omega,\mu) \in \mathbb{R}^d \times \mathbb{T} \mathbb{S}^{d-1}.
\end{align*}
Now by condition \eqref{tangent_condition_tangentbundle}, or see also \eqref{Condition_TS1}, the unknown $G\in C^\infty(\mathbb{R}^d \times \mathbb{T} \mathbb{S}^{d-1})$ has to be chosen in order that 
\begin{align*}
\begin{pmatrix} -(I-\omega \otimes \omega) \nabla \Phi(\xi) \\ G(\xi,\omega,\mu) \end{pmatrix} \cdot \begin{pmatrix} \mu \\ \omega \end{pmatrix} =0,~(\xi,\omega,\mu) \in \mathbb{R}^d \times \mathbb{T} \mathbb{S}^{d-1}.
\end{align*}
This means $G(\xi,\omega,\mu) \cdot \omega = \nabla \Phi(\xi) \cdot \mu$ for $(\xi,\omega,\mu) \in \mathbb{R}^d \times \mathbb{T} \mathbb{S}^{d-1}$ which is satisfied by setting $G(\xi,\omega,\mu):= \mu \cdot \nabla \Phi(\xi) ~ \omega$ for each $(\xi,\omega,\mu) \in \mathbb{R}^d \times \mathbb{T} \mathbb{S}^{d-1}$. So the smooth fiber lay-down model is defined with help of the following Stratonovich SDE as
\begin{align} \label{smoothfiber_lay_down_equation}
&\mathrm{d}\xi = \omega \, \mathrm{dt} \\
&\mathrm{d}\omega = - (I-\omega \otimes \omega)  \nabla \Phi ( \xi) \, \mathrm{dt} + \mu \, \mathrm{dt} \nonumber \\
&\mathrm{d}\mu = \mu \cdot \nabla \Phi(\xi) ~ \omega  \,\mathrm{dt} - \lambda \, \mu  \,\mathrm{dt} -|\mu|^2 \omega  \, \mathrm{dt} + \sigma \, (I-\omega \otimes \omega) \circ \mathrm{d}W_t \nonumber
\end{align}
with state space $\mathbb{R}^d \times \mathbb{T} \mathbb{S}^{d-1}$. Here $W$ is a standard $d$-dimensional Brownian motion and $\lambda, \sigma$ are nonnegative constants. For notational convenience, the time index $t$ is omitted in the SDE. As already mentioned in the introduction, the model has been developed in cooperation with the authors from \cite{KMW12B}. It has first been detected in its full form including all necessary terms by the second named author of the underlying article. 

\begin{Rm}
Alternatively, one can proceed as follows. An equivalent form of the two-dimensional basic fiber lay-down equation \eqref{SDE_Basic_Fiber_Lay_Down} is given by 
\begin{align} \label{eq_2d_model_basic_alpha}
&\mathrm{d}\xi = \tau(\theta) \, \mathrm{dt}\\
&\mathrm{d}\theta = - \tau^\bot(\theta) \cdot \nabla \Phi(\xi) \, \mathrm{dt} + \sigma \,\mathrm{d}W_t \nonumber
\end{align}
with $W$ being a one-dimensional standard Brownian motion, $\tau(\theta):=(\cos(\theta),\sin(\theta))^T$, $\tau^\bot :=\frac{ \partial \tau}{\partial \theta}$ and $\Phi \in C^\infty(\mathbb{R}^2)$. The equation has more precisely to be understood as an abstract Stratonovich SDE with state space $\mathbb{R}^2 \times \mathbb{R} / {2 \pi \mathbb{Z}}$ and equivalent means here that any solution coincides weakly with any solution to \eqref{SDE_Basic_Fiber_Lay_Down} (for $d=2$), see \cite{GKMS12} for details. Now a natural smoother version of \eqref{eq_2d_model_basic_alpha} reads
\begin{align} \label{eq_2d_model_smooth_alpha}
&\mathrm{d}\xi = \tau(\theta) \, \mathrm{dt} \\
&\mathrm{d}\theta = - \tau^\bot(\theta) \cdot \nabla \Phi(\xi) \, \mathrm{dt} + \nu \,\mathrm{dt} \nonumber\\
&\mathrm{d}\nu = -\lambda \,\nu \mathrm{dt} + \sigma\, \mathrm{d}W_t \nonumber
\end{align}
with state space $\mathbb{R}^2 \times \mathbb{R} / {2 \pi \mathbb{Z}} \times \mathbb{R}$ and $\lambda$ some nonnegative constant. To derive a natural higher dimensional version of the latter one proceeds exactly as in the Langevin case, see Section \ref{Derivation_Langevin_process}. Therefore, consider the diffeomorphism 
\begin{align*}
\mathbb{R}^2 \times \mathbb{R} / {2 \pi \mathbb{Z}} \times \mathbb{R} \ni (\xi,\theta,\nu) \mapsto (\xi, \omega,\mu)=\begin{pmatrix} \xi \\ \tau(\theta) \\ \nu \, \tau^\bot(\theta) \end{pmatrix}
\end{align*}
with image $\mathbb{R}^2 \times \mathbb{T}\mathbb{S}^1$. Now by following identically the argumentation from Section \ref{section_langevin_sphere} and by additionaly calculating the pushforward vector field of  $-\tau^\bot \cdot \nabla \Phi \,\frac{\partial}{\partial \theta}$ simply as $ \begin{pmatrix} -(I-\omega \otimes \omega)  \nabla \Phi \\ \mu \cdot \nabla \Phi ~ \omega \end{pmatrix}$, we obtain that the natural higher dimensional version of \eqref{eq_2d_model_smooth_alpha} is given by \eqref{smoothfiber_lay_down_equation}. In this sense both mathematical derivations are consistent and commutative.
\end{Rm}

Nevertheless, Equation \eqref{smoothfiber_lay_down_equation} is not yet in suitable form for numerical simulations. As described in Remark \ref{Rm_simulating_geometric_Langevin}, we need to derive a so called local SDE. For the original basic fiber lay-down equation \eqref{SDE_Basic_Fiber_Lay_Down} this is done in \cite{GKMS12} in each dimension. Therefore, we shall also calculate such equations for the smooth model from above for each $d \in \mathbb{N}$, $d \geq 2$. Moreover, this then also gives a geometrically precise justification of the calculations done in \cite{KMW12B} and contains the practical relevant smooth three-dimensional fiber lay-down model in local form as special case. The proof partly follows the one from Proposition \ref{Pp_Langevin_generator}.

Therefore, we introduce first the following special local parametrization of $\mathbb{S}^{n}$, $n \in \mathbb{N}$, given by $\tau^{(n)} :\mathbb{R} / {2 \pi \mathbb{Z}} \times (0,\pi)^{n-1}  \rightarrow \mathbb{S}^n$ with $\tau^{(1)}(\theta_1):=\left(\cos(\theta_1),\sin(\theta_1)\right)^T$, $\theta_1 \in \mathbb{R} / {2 \pi \mathbb{Z}}$, and inductively 
\begin{align*}
\tau^{(n)}(\theta_1,\ldots \theta_n):=\begin{pmatrix} \tau^{(n-1)}(\theta_1,\ldots,\theta_{n-1}) \sin(\theta_n) \\ \cos(\theta_n) \end{pmatrix},~n \in \mathbb{N},~n \geq 2.
\end{align*}
In the following we write $\tau$ instead of $\tau^{(n)}$. Note that $g_{ij}=0$ for $i \not= j$. Furthermore, the spherical unit vectors on $\mathbb{S}^n$, $n \in \mathbb{N}$, are defined by 
\begin{align} \label{eq_spherical_uni_vectors}
n_j=\frac{1}{|\partial_{\theta_j} \tau |} \frac{\partial \tau}{\partial \theta_j} = \sqrt{g^{jj}} \,\frac{\partial \tau}{\partial \theta_j},~j=1,\ldots, n.
\end{align}
For the rest of this section, let $\mathbb{Y}:=\mathbb{R}^d \times \mathbb{R} / {2 \pi \mathbb{Z}} \times (0,\pi)^{\,d-2} \times \mathbb{R}^{d-1}$ with $d$ as before and $\mathbb{Y}=\mathbb{R}^2 \times \mathbb{R} / {2 \pi \mathbb{Z}} \times \mathbb{R}$ in case $d=2$. Then a local parametrization of $\mathbb{R}^d \times \mathbb{T}\mathbb{S}^{d-1}$ is given by
\begin{align*}
\mathbb{Y} \ni (\xi,\theta,\nu) \mapsto (\xi,\omega,\mu):=\begin{pmatrix} \xi\\ \tau(\theta) \\ \sum_{j} \nu_j \, n_j(\theta) \end{pmatrix}.
\end{align*}
Finally, define 
\begin{align*}
\Delta_{inj}:=\sqrt{g^{ii}}\,\partial_{\theta_i} n_n \cdot n_j
\end{align*}
for $i,n,j=1,\ldots,d-1$. We arrive at the desired result. Recall our definition of a local SDE, see Remark \ref{Rm_simulating_geometric_Langevin}, and let $L : C^\infty(\mathbb{R}^d \times \mathbb{T}\mathbb{S}^{d-1}) \rightarrow C^\infty(\mathbb{R}^d \times \mathbb{T}\mathbb{S}^{d-1})$ be the Kolmogorov operator associated to the smooth fiber lay-down model. 

\begin{Pp} \label{Pp_local_SDE_smooth_model}
A local SDE on $\mathbb{Y}$ associated to the smooth fiber lay-down SDE \eqref{smoothfiber_lay_down_equation} generating also a (local) $L$-diffusion process is given by
\begin{align} \label{local_SDE_smooth_model}
&\mathrm{d}\xi = \tau(\theta)\, \mathrm{dt} \\ 
&\mathrm{d}\theta_j = - \sqrt{g^{jj}} \,\nabla \Phi(\xi) \cdot n_j(\theta) \, \mathrm{dt}  + \sqrt{g^{jj}} \,\nu_j \, \mathrm{dt} \nonumber\\
&\mathrm{d}\nu_j = \sum_{i,n} \Delta_{inj} \, \nabla \Phi(\xi) \cdot n_i(\theta) ~\nu_n \, \mathrm{dt} - \sum_{i,n} \Delta_{inj} \, \nu_i \,\nu_n \,\mathrm{dt} - \lambda\,\nu_j \,\mathrm{dt} + \sigma\, \mathrm{d}W_t^{(j)} \nonumber
\end{align}
where $j=1,\ldots,d-1$. $W$ is a standard $(d-1)$-dimensional Brownian motion and $\Phi \in C^\infty(\mathbb{R}^d)$. In particular, for $d=2$ Equation \eqref{local_SDE_smooth_model} reduces to \eqref{eq_2d_model_smooth_alpha} and for $d=3$ we obtain the three-dimensional local SDE for the smooth fiber lay-down model.
\end{Pp}

\begin{proof}
At first, we get the pushforward vector fields
\begin{align*}
\widetilde{\frac{\partial}{\partial \xi_i}} \equiv \begin{pmatrix} e_i\\ 0 \\ 0\end{pmatrix},~\widetilde{\frac{\partial}{\partial \theta_j}} \equiv \begin{pmatrix} 0\\ \partial_{\theta_j} \tau \\ \sum_{i} \nu_i\, \partial_{\theta_j}  n_i \end{pmatrix},~\widetilde{\frac{\partial}{\partial \nu_j}} \equiv \begin{pmatrix} 0 \\ 0\\  n_j \end{pmatrix}.
\end{align*}
where $e_i$ is the $i$-th unit vector in $\mathbb{R}^d$. Now there are uniqely determined $a_i,b_j,c_j \in \mathbb{R}$, all depending on $(\xi,\omega,\nu)$, such that
\begin{align*}
\sum_i a_i \, \widetilde{\frac{\partial}{\partial \xi_i}} + \sum_j b_j \widetilde{\frac{\partial}{\partial \theta_j}} + \sum_j c_j \widetilde{\frac{\partial}{\partial \nu_j}} \equiv \mathcal{A}(\xi,\omega,\mu).
\end{align*}
Clearly $a_i=\tau_i$ for all $i=1,\ldots,d$. And since $(I-\omega \otimes \omega) \nabla \Phi(\xi)=\sum_j \nabla \Phi(\xi) \cdot n_j \,n_j$ we obtain $b_j=-\sqrt{g^{jj}}\,\nabla \Phi(\xi) \cdot n_j$ for all $j=1,\dots,d-1$. By taking the scalar product with respect to some $\widetilde{\frac{\partial}{\partial \nu_j}}$ on both sides of the last equation this yields
\begin{align*}
c_j = - \sum_i b_i \, \sum_n \nu_n \, \partial_{\theta_i}  n_n \cdot n_j = \sum_{i,n} \Delta_{inj} \, \nabla \Phi(\xi) \cdot n_i ~\nu_n.
\end{align*}
Next, the previous arguments can be repeated with redefined $a_i,b_j,c_j$ for the smooth vector field from $\mathbb{R}^d \times \mathbb{T} \mathbb{S}^{d-1}$ given by
\begin{align*}
\begin{pmatrix} 0 \\ \mu \\ -|\mu|^2 \omega \end{pmatrix},~(\xi,\omega,\mu) \in \mathbb{R}^d \times \mathbb{T} \mathbb{S}^{d-1}.
\end{align*}
Then $a_i=0$ for all $i=1,\ldots,d,$ and $b_j=\sqrt{g^{jj}}\,\nu_j$ for all $j=1,\ldots,d-1$. Hence $c_j = - \sum_{i,n} \Delta_{inj} \, \nu_i \, \nu_n$. Note that the remaining vector fields in the smooth fiber lay down-model are already computed in the proof of Proposition \ref{Pp_Langevin_generator}. Following the notation from the latter, we have
\begin{align*}
\kappa_j = \sqrt{g^{jj}} \nu_j,~\frac{\partial}{\partial \nu_j} = \sqrt{g^{jj}}\, \frac{\partial}{\partial \kappa_j}.
\end{align*}
This implies
\begin{align*}
\kappa_j \, \frac{\partial}{\partial \kappa_j} = \nu_j \, \frac{\partial}{\partial \nu_j},~g^{jj} \, \frac{\partial^2}{\partial \kappa_j^2} = \frac{\partial^2}{\partial \nu_j^2}.
\end{align*}
Altogether, the generator associated to the local SDE from the proposition coincides with the generator of the smooth fiber lay-down model computed on $\mathbb{Y}$.
\end{proof}

\begin{figure}[tbp] 
\subfigure{\includegraphics[scale=0.28]{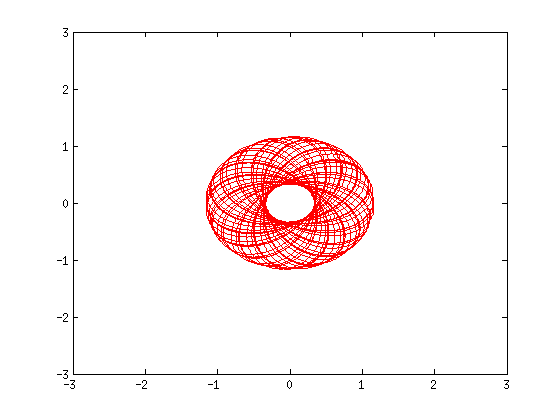}}
\subfigure{\includegraphics[scale=0.28]{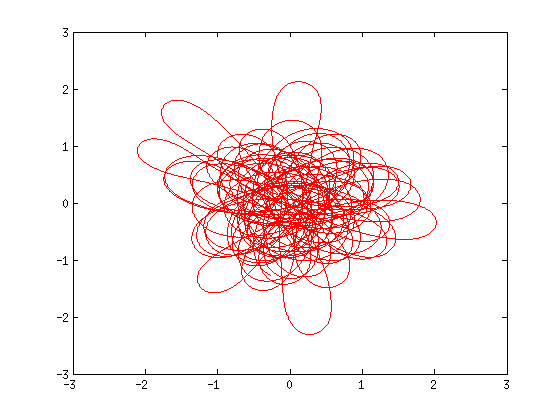}}
\subfigure{\includegraphics[scale=0.28]{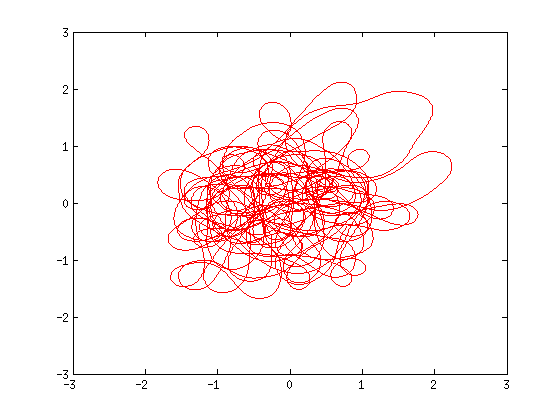}}
\subfigure{\includegraphics[scale=0.28]{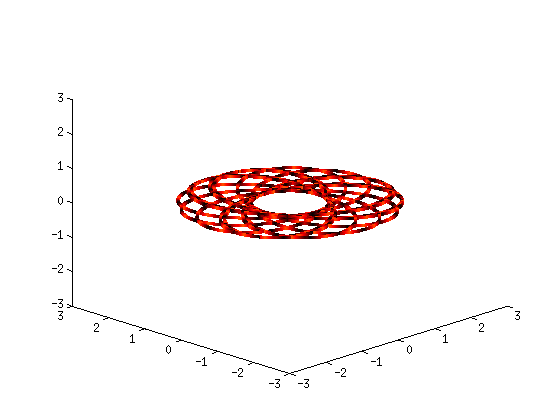}}
\subfigure{\includegraphics[scale=0.28]{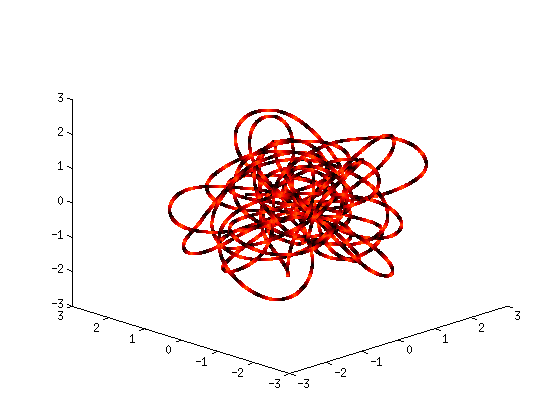}}
\subfigure{\includegraphics[scale=0.28]{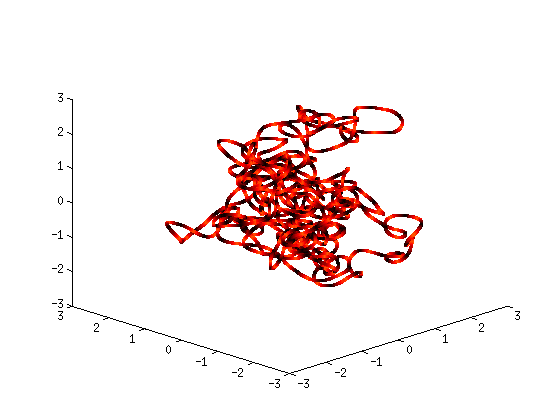}}
\caption{Smooth fiber-lay down} \label{figure_FiberSmooth}
\end{figure}

In Figure \ref{figure_FiberSmooth} we plot examplified the $\xi$-trajectory of the two- and three-dimensional smooth fiber lay down model via simulating Equation \eqref{local_SDE_smooth_model}. Therefore, for $d=3$ one easily calculates $\Delta_{112}=-\cot\theta_2$, $\Delta_{121}=\cot\theta_2$ and $0$ else. The simulations are obtained by setting $\lambda=1$, $\Phi = |\xi|^2$ and $\sigma=0$, $\sigma=0.5$, $\sigma=3$. Finally, we remark that in analogy to the basic fiber lay-down model, the stationary distribution of the associated Fokker-Planck equation corresponding to the smooth fiber lay-down model is again explicitly known, see \cite{KMW12B}. Parameters as well as the potential can then be chosen such that the stationary distribution coincides with the real fiber-distribution as observed in a realistic industrial fiber lay-down process.

\begin{Rm}
We do not aim here to discuss the physical relevant scenarios for the choice of $\Phi$, $\sigma$ and $\lambda$ in case of a realistic fiber lay-down process as observed in the production process of nonwovens. Instead, for the smooth model we refer to \cite{KMW12B}. See also \cite{KMW12} as well as \cite{KMW09} and references therein. Moreover, we remark that an additional parameter $B$ can be included into the three-dimensional model for measuring the anisotropic orientation of the fibers. This can be done as presented in \cite{KMW12} and \cite{KMW12B}. An industrial relevant discussion of the existing (basic and smooth) fiber lay-down models is planned in a forthcoming research article of the authors from \cite{KMW12B} together with the authors of the underlying paper.
\end{Rm}

\subsection{Self-propelled interacting particle systems with roosting} 

In \cite{CKMT10} a two-dimensional spherical velocity model about self-propelled interacting particle systems with roosting force describing the collective behavior of swarms of animals is derived. The equation reads (at first informally) as
\begin{align*}
&\mathrm{d}\xi_i = r \,\tau(\theta_i) \, \mathrm{dt} \\ 
&\mathrm{d}\theta_i = -\frac{1}{r} \,\tau^\bot(\theta_i) \cdot \nabla_{\xi_i} \Phi (\xi_i) \, \mathrm{dt} - \frac{1}{K} \frac{1}{r} \, \tau^\bot(\theta_i) \cdot \nabla_{\xi_i} \sum_{j\not=i} U(|\xi_j-\xi_i|) \, \mathrm{dt} + \frac{\sigma}{r} \,\mathrm{d}W_{i,t}.
\end{align*}
The index $i$ means the $i$-th particle with $i=1,\ldots,K,$ $K \in \mathbb{N}$ and $W=(W_1,\ldots,W_K)$ is a $K$-dimensional standard Brownian motion. Here $\tau$ and $\tau^\bot$ are defined as above,  $r \in (0,\infty)$ is the constant velocity, $\sigma \in (0,\infty)$ the diffusion constant, $\Phi$ and $V$ are suitable potential functions. For details, derivation and motivation we refer to \cite{CKMT10}. Via defining
\begin{align*}
\Psi_i(\xi_i):=\Phi(\xi_i) + \frac{1}{K} \sum_{j \not= i} U(|\xi_j-\xi_i|),~i=1,\ldots,K,
\end{align*}
the two-dimensional model reads
\begin{align} \label{eq_self_propelled_2d}
&\mathrm{d}\xi_i = r \,\tau(\theta_i) \, \mathrm{dt} \\ 
&\mathrm{d}\theta_i = -\frac{1}{r} \,\tau^\bot(\theta_i) \cdot \nabla_{\xi_i} \Psi(\xi_i) \, \mathrm{dt} + \frac{\sigma}{r} \,\mathrm{d}W_{i,t} \nonumber
\end{align}
where $i=1,\ldots,K$. So the model has an analogous form as the original two-dimensional fiber lay-down model and should precisely be understood as an abstract Stratonovich SDE with state space $(\mathbb{R}^2 \times \mathbb{R} / {2 \pi \mathbb{Z}})^K$, see \cite{GKMS12} for the differential geometric details. Now by following again \cite{GKMS12}, the higher dimensional (coordinate free) version of this model has state space $(\mathbb{R}^d \times \mathbb{S}_r)^K$, $d \in \mathbb{N}$, $d \geq 2$, and is simply given by
\begin{align} \label{eq_self_propelled_general}
&\mathrm{d}\xi_i = \omega_i \,\mathrm{dt} \\
&\mathrm{d}\omega_i = - \Pi_{\mathbb{S}_r}[\omega_i] \,\nabla_{\xi_i} \Psi(\xi_i) \, \mathrm{dt} + \Pi_{\mathbb{S}_r}[\omega_i] \circ \mathrm{d}W_{i,t} \nonumber
\end{align}
where $i=1,\ldots,K$. Here $W=(W_1,\ldots,W_K)^T$ is a $d \cdot K$-dimensional standard Brownian motion and $\mathbb{S}_r$ denotes the sphere of radius $r>0$ in $\mathbb{R}^d$. In case $d=2$, Equation \eqref{eq_self_propelled_general} reduces to the two-dimensional model \eqref{eq_self_propelled_2d}, see \cite{GKMS12} for details (and compare this also with the first example in Section \ref{Examples_My_geometric_Langevin_equation}). For $d=3$ one gets a physical relevant three-dimensional model, see \cite{GKMS12} for an explicit numerical simulation formula. Moreover, in  exactly the same way as in our smooth fiber lay-down model from Section \ref{section_smooth_fiber_lay_down}, one obtains a smoother version of \eqref{eq_self_propelled_general} now in the obvious way. Its applicability for physical relevant situations has of course to be checked and is left for future research.

\section{The geometric Langevin equation with spherical velocity} \label{Langevin_process_constant_velocity_version}
Next we aim to derive a natural analogue of the geometric Langevin process \eqref{Langevin_equation_geometric} moving with velocity of constant absolute value, i.e., $|\omega|=r$, $r > 0$. This is motivated by the mathematical structure detected in the fiber lay-down model in Section \ref{Section_Geometry_of_Fiber_Lay_Down}. In order to derive this Langevin model with velocity of constant absolute value, we follow the strategy as described in the fiber lay-down case in Section \ref{Section_Geometry_of_Fiber_Lay_Down} and project the coordinates $\omega$ in the geometric Langevin equation $\eqref{Langevin_equation_geometric}$ onto the sphere of radius $r$ given by $\{ \omega~|~|\omega|^2=r^2\}$. Then the resulting process should have state space $\mathbb{S}_r\mathbb{M}$. Here recall that $\mathbb{S}_r\mathbb{M}$ denotes the spherical tangent bundle bundle of $\mathbb{M}$ of radius $r$. In the following, we realize this equation in mathematical precise form.

Therefore, note first that the orthogonal projection from $\mathbb{R}^N$ onto the submanifold $\mathbb{S}_r^{N-1}=\{\omega \in \mathbb{R}^{N}~|~|\omega|^2=r^2\}$ is given by 
\begin{align*}
\Pi_{\mathbb{S}_r^{N-1}}[\omega]=I- \frac{1}{r^2 }\,\omega \otimes \omega,~\omega \in \mathbb{S}_r^{N-1}.
\end{align*}
We abbreviate $\mathbb{S}_r^{N-1}$ by $\mathbb{S}_r$ and for $r=1$ we shortly write $\mathbb{S}$ instead of $\mathbb{S}_1$. By the discussion above, it is natural to propose the following Stratonovich SDE
\begin{align} \label{Langevin_equation_geometric_constant_speed}
&\mathrm{d} \xi_t = \omega_t \, \mathrm{dt} \\
&\mathrm{d} \omega_t = \Pi_{\mathbb{S}_r}[\omega_t]\Big(-\lambda \,\omega_t - F(\xi_t,\omega_t)  - \text{grad}_{\mathbb{M}} \, \Phi (\xi_t)\Big)\,\mathrm{dt} + \sigma \, \Big(\Pi_{\mathbb{S}_r}[\omega_t]  \,\Pi_{\mathbb{M}}[\xi_t] \Big)\circ \mathrm{d}W_t. \nonumber
\end{align}
with some $N$-dimensional standard Brownian motion $W$ and state space $\mathbb{S}_r\mathbb{M}$, $r>0$. Of course, this equation must be understood more precisely as an abstract Stratonovich equation similiar as the geometric Langevin equation, see \eqref{Langevin_equation_geometric}. In the equation above $\Pi_{\mathbb{S}_r}[\omega] \, \Pi_{\mathbb{M}}[\xi]$ denotes the usual matrix product between $\Pi_{\mathbb{S}_r}[\omega]$ and $\Pi_{\mathbb{M}}[\xi]$. $F(\xi,\omega)$ is defined in \eqref{Definition_forcing_term_F} and $\Phi \in C^\infty(\mathbb{M})$. One easily verifies that the vector fields involved in \eqref{Langevin_equation_geometric_constant_speed} are really tangential to $\mathbb{S}_r\mathbb{M}$. Indeed, observe that given a vector field $\mathcal{A}$ satisfying condition \eqref{tangent_condition_tangentbundle}, then the vector field $\mathcal{B}$ with
\begin{align*}
\mathcal{B}(\xi,\omega):=\begin{pmatrix} I & 0 \\ 0 & \Pi_{\mathbb{S}_r}[\omega] \end{pmatrix} \mathcal{A}(\xi,\omega),~(\xi,\omega) \in \mathbb{S}_r\mathbb{M},
\end{align*}
fulfills conditions \eqref{tangent_condition_tangentbundle} and \eqref{tangent_condition_unittangentbundle}. Furthermore, we have $\Pi_{\mathbb{S}_r}[\omega]\left( \lambda \,\omega \right) =0$ as well as
\begin{align*}
\Pi_{\mathbb{S}_r}[\omega] J_\xi^T x  = J_\xi^T x - \frac{\omega}{r^2} \left( \omega, J_\xi^T x \right)_{\text{euc}}  = J_\xi^T x - \frac{\omega}{r^2} \left( J_\xi \, \omega, x \right)_{\text{euc}}  = J_\xi^T x
\end{align*}
for all $(\xi,\omega) \in \mathbb{S}_r\mathbb{M}$ and each $x \in \mathbb{R}^k$. Hence we have $\Pi_{\mathbb{S}_r}[\omega] F(\xi,\omega) = F(\xi,\omega)$ for each $(\xi,\omega) \in \mathbb{S}_r\mathbb{M}$. Moreover, recall the definition of $\text{grad}_{\mathbb{M}} \, \Phi (\xi)$ from \eqref{Def_gradient_submanifold}. Altogether, the analogue manifold-valued Stratonovich SDE for the geometric Langevin process moving with velocity of constant absolute value simply reduces to
\begin{align} \label{Langevin_equation_geometric_constant_speed_final}
&\mathrm{d} \xi_t = \omega_t \, \mathrm{dt} \\
&\mathrm{d} \omega_t =  - F(\xi_t,\omega_t) \, \mathrm{dt}   - \Big(\Pi_{\mathbb{S}_r}[\omega_t] \Pi_{\mathbb{M}}[\xi_t] \Big)\nabla \Phi (\xi_t) \,\mathrm{dt} + \sigma \, \Big(\Pi_{\mathbb{S}_r}[\omega_t]  \,\Pi_{\mathbb{M}}[\xi_t] \Big)\circ \mathrm{d}W_t \nonumber
\end{align}
and has state space $\mathbb{S}_r\mathbb{M}$.  We call Equation \eqref{Langevin_equation_geometric_constant_speed_final} the \textit{geometric Langevin equation with spherical velocity}. Note that in case $\mathbb{M}=\mathbb{R}^d$ it comes out the fiber lay-down Stratonovich SDE \eqref{SDE_Basic_Fiber_Lay_Down}, see Section \ref{Section_Geometry_of_Fiber_Lay_Down}. Here recall the identification made in \eqref{identification_Rd}. Further examples of this process are presented in the last section in order to convince the reader from this beauty geometric nature of the stochastic kinetic equation and to give further applications.

\begin{Rm}
It is fascinating to note that this equation arised from industrial real world applications. Indeed, the second named author of this article got this idea via trying to understand the geometry occuring behind the fiber lay-down world. 
\end{Rm}

\section{Examples of the geometric Langevin process with spherical velocity} \label{Examples_My_geometric_Langevin_equation}

In this section we provide some concret examples of our geometric Langevin equation moving with velocity of constant absolute value. At the one hand, this visualizes the dynamics and shows its wonderful geometric nature. At the other hand, the second example below can be used in specific applications. Therein, we  discuss the geometric Langevin process with spherical velocity on $\mathbb{S}_r\mathbb{M}$ where $\mathbb{M}=\mathbb{S}^2$. Analogously to the spherical Langevin process discussed in Section \ref{Spherical_Langevin}, this stochastic process (or more precisely its $\xi$-coordinates) serves as a further smooth process on $\mathbb{S}^2$. In particular, it can also be used to extend stochastic kinetic models having to rough paths due to an appearing Brownian motion on $\mathbb{S}^2$. A concrete example of this is again our basic (three-dimensional) fiber lay-down model, see Section \ref{Section_Geometry_of_Fiber_Lay_Down}. Consequently, a further smooth three-dimensional fiber lay-down model can be constructed. As mentioned in the introduction, we intend to publish an additional research article concerning this application.

\subsection{The cylinder \texorpdfstring{$\mathbb{M}=\mathbb{S}^1 \times \mathbb{R}$}{}}

We start with the cylinder $\mathbb{M}=\mathbb{S}^1 \times \mathbb{R}$. Thus
\begin{align*}
\mathbb{S}_r\mathbb{M}= \{ (\xi,\omega) \in \mathbb{R}^3 \times \mathbb{R}^3~|~\xi_1^2 + \xi_2^2=1,~\begin{pmatrix} \xi_1 \\ \xi_2 \end{pmatrix} \cdot \begin{pmatrix} \omega_1 \\ \omega_2 \end{pmatrix} = 0 ,~|\omega|^2=r^2\},~r > 0.
\end{align*}
In this section let $\mathbb{Y}:=\mathbb{R} / {2 \pi \mathbb{Z}}$. There exists a natural diffeomorphism between $\mathbb{Y} \times \mathbb{R} \times \mathbb{Y}$ and $\mathbb{S}_r\mathbb{M}$ given by
\begin{align} \label{state_space_transform_2d_fiber}
(\alpha,\xi_3,\theta) \mapsto \begin{pmatrix} \xi, \omega \end{pmatrix} =\begin{pmatrix} \varphi\\ \xi_3 \\ r\tau_{1} \,\varphi^\bot \\ r\tau_{2} \end{pmatrix},~\varphi(\alpha):=\begin{pmatrix} \cos\alpha \\ \sin\alpha \end{pmatrix},~\tau(\theta):=\begin{pmatrix} \cos\theta \\ \sin\theta \end{pmatrix}
\end{align}
Here $\bot$ denotes the first derivative w.r.t.~$\theta$ or $\alpha$ respectively.  For abuse of notation we omit sometimes the arguments of the occuring functions as well as the time index $t$. We have the following result. The proof works similiar as the one from the upcoming result in our second example, see the Appendix of this section.

\begin{Lm} \label{Lm_generator_geometricLangevin_cylinder}
The generator $L:C^\infty(\mathbb{S}_r\mathbb{M}) \rightarrow C^\infty(\mathbb{S}_r\mathbb{M})$ with $r>0$ of the geometric Langevin equation with spherical velocity on $\mathbb{S}_r\mathbb{M}$, $\mathbb{M}=\mathbb{S}^1 \times \mathbb{R}$, is given on $\mathbb{Y} \times \mathbb{R} \times \mathbb{Y}$ as
\begin{align*}
L=r \,\tau_1 \, \frac{\partial}{\partial \alpha} + r \,\tau_2 \, \frac{\partial}{\partial \xi_3} + \frac{1}{r} \,\tau^\bot \cdot \nabla_{(\alpha,\xi_3)} \Psi  \,\frac{\partial}{\partial \theta} + \frac{\sigma^2}{2r^2} \,\frac{\partial^2}{\partial \theta^2}.
\end{align*}
Here $\Psi(\alpha,\xi_3):=\Phi(\varphi(\alpha),\xi_3)$,~$(\alpha,\xi_3) \in \mathbb{Y} \times \mathbb{R}$ for $\Phi \in C^\infty(\mathbb{M})$.
\end{Lm}
Consequently, an equivalent SDE on $\mathbb{Y} \times \mathbb{R} \times \mathbb{Y}$ generating the (weakly unique) $L$-diffusion process reads
\begin{align} \label{equivalent_SDE_constantvelociy_cylinder}
&\mathrm{d} (\alpha, \xi_3)^T = r \,\tau(\theta) \, \mathrm{dt}\\
&\mathrm{d} \theta = \frac{1}{r}\, \tau^\bot(\theta) \cdot \nabla \Psi (\alpha,\xi_3) \, \mathrm{dt} + \frac{\sigma}{r} \,\mathrm{d}W_t. \nonumber
\end{align}

\begin{Rm}
This result is not surprising, having the explicit two-dimensional fiber lay-down equation in mind, see e.g.~Section \ref{Section_Applications} or \cite{GKMS12}. Assuming a periodic boundary on the latter in the first variable, we obtain \eqref{equivalent_SDE_constantvelociy_cylinder}. This even shows that we can alternatively observe the general structure of the spherical velocity version of the geometric Langevin equation at our original fiber lay-down equation with the help of the previously defined state space transformation \eqref{state_space_transform_2d_fiber}.
\end{Rm}

In Figure \ref{figure_geometric_Langevin_Speed_cylinder} we plot the $\xi$-coordinates of the geometric Langevin equation with spherical velocity on $\mathbb{S}\mathbb{M}$ ($r=1$) via simulating \eqref{equivalent_SDE_constantvelociy_cylinder}. We choose $\Phi=0$ and $\sigma=0$ on the left hand side and $\sigma=4$ on the right hand side.

\begin{figure}[tbp] 
\subfigure{\includegraphics[scale=0.41]{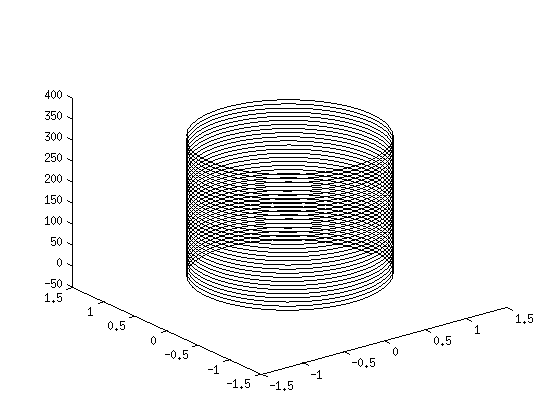}}
\subfigure{\includegraphics[scale=0.41]{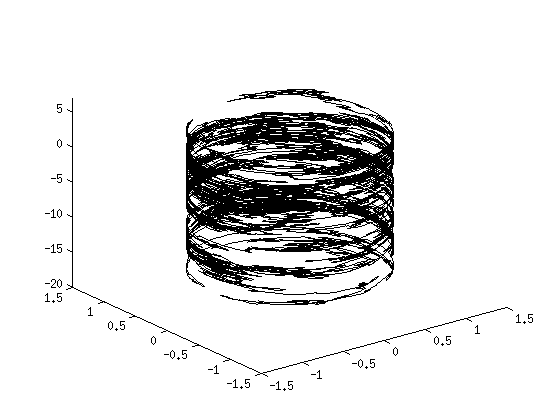}}
\caption{Langevin process with spherical velocity on $\mathbb{S}^1 \times \mathbb{R}$} \label{figure_geometric_Langevin_Speed_cylinder}
\end{figure}

\subsection{The sphere \texorpdfstring{$\mathbb{M}=\mathbb{S}^2$}{}} Now we set $\mathbb{M}:= \mathbb{S}^2$. Then 
\begin{align*}
\mathbb{S}_r\mathbb{M} = \{ (\xi,\omega) \in \mathbb{R}^3 \times \mathbb{R}^3~|~|\xi|^2=1,~\xi \cdot \omega=0,~|\omega|^2=r^2\},~r >0.
\end{align*}
A local parametrization of $\mathbb{S}_r\mathbb{M}$ is given on $\mathbb{Y} \times (0,\pi) \times \mathbb{Y}$, $\mathbb{Y}:=\mathbb{R} / {2 \pi \mathbb{Z}}$, as
\begin{align*}
\mathbb{Y} \times (0,\pi) \times \mathbb{Y} \ni (\theta_1,\theta_2,\alpha) \mapsto \begin{pmatrix} \xi, \omega \end{pmatrix} = \begin{pmatrix} \tau \\ r \varphi_1\, n_1 + r \varphi_2 \, n_2 \end{pmatrix}.
\end{align*}
In this example $\tau$ and $\varphi$ are defined as
\begin{align*} 
\tau(\theta_1,\theta_2):= \left(\cos\theta_1 \sin\theta_2, \sin\theta_1 \sin\theta_2, \cos\theta_2 \right)^T,~\varphi(\alpha):=\left( \cos\alpha, \sin\alpha \right)^T
\end{align*}
and $n_1=n_1(\theta_1)$, $n_2=n_2(\theta_1,\theta_2)$ are the spherical unit vectors $n_1 := |\partial_{\theta_1} \tau|^{-1} \partial_{\theta_1} \tau$, $n_2 := \partial_{\theta_2} \tau$, see Equation \eqref{eq_spherical_uni_vectors}. We have the following result, for the proof we refer to the appendix of this section.

\begin{Lm} \label{Lm_generator_geometricLangevin_sphere}
The generator $L:C^\infty(\mathbb{S}_r\mathbb{M}) \rightarrow C^\infty(\mathbb{S}_r\mathbb{M})$ with $r>0$ of the spherical velocity version of the geometric Langevin equation with state space $\mathbb{S}_r\mathbb{M}$, $\mathbb{M}=\mathbb{S}^2$, is given locally on $\mathbb{Y} \times (0,\pi) \times \mathbb{Y}$ as
\begin{align} \label{generator_sphere_geometric_langevin_speed}
L=&~r \frac{\cos \alpha}{\sin\theta_2} \,\frac{\partial}{\partial{\theta_1}} + r \sin\alpha \,\frac{\partial}{\partial{\theta_2}}  \\
~&+ \left( r\cos \alpha \cot\theta_2 + \frac{1}{r}\frac{\sin \alpha}{\sin(\theta_2)}\,\partial_{\theta_1} \Psi - \frac{1}{r} \cos \alpha \,\partial_{\theta_2} \Psi \right) \frac{\partial}{\partial{\alpha}}+ \frac{\sigma^2}{2r^2}\frac{\partial^2}{\partial{\alpha}^2}. \nonumber
\end{align}
In this case $\Psi(\theta_1,\theta_2):=\Phi( \tau(\theta_1,\theta_2))$,~$(\theta_1,\theta_2) \in \mathbb{Y} \times (0,\pi)$, where $\Phi \in C^\infty(\mathbb{M})$.
\end{Lm}

In particular, an SDE modelling the (local) $L$-diffusion on $\mathbb{Y} \times (0,\pi) \times \mathbb{Y}$ reads
\begin{align} \label{equivalent_SDE_constantvelociy_sphere}
&\mathrm{d} \theta_1 = r\frac{\cos \alpha}{\sin \theta_2} \, \mathrm{dt}\\
&\mathrm{d} \theta_2 = r\sin \alpha \, \mathrm{dt} \nonumber \\
&\mathrm{d} \alpha =  r\cos \alpha \cot\theta_2\, \mathrm{dt} + \frac{1}{r}\frac{\sin \alpha}{\sin \theta_2}\,\partial_{\theta_1} \Psi \, \mathrm{dt} - \frac{1}{r}\cos \alpha \,\partial_{\theta_2} \Psi  \,\mathrm{dt} + \frac{\sigma}{r} \,\mathrm{d}W_t. \nonumber
\end{align}

\begin{Rm}
Note that a stationary solution to the (formal) Fokker-Planck equation associated to \eqref{equivalent_SDE_constantvelociy_sphere} is given on $\mathbb{Y} \times (0,\pi) \times \mathbb{Y}$ by $e^{-\frac{\Psi}{r^2}} \sin \theta_2$. This form is expected having the well-known form of the stationary solution to the fiber-lay down Fokker-Planck equation in mind, see Section \ref{Section_Geometry_of_Fiber_Lay_Down}.
\end{Rm}

In Figure \ref{figure_geometric_Langevin_Speed_sphere} we plot the $\xi$-coordinates of the geometric Langevin equation with spherical velocity on $\mathbb{S}\mathbb{M}$ ($r=1$) via simulating Equation \eqref{equivalent_SDE_constantvelociy_sphere}. $\sigma$ increases from $\sigma=0,~\sigma=0.1$, $\sigma=1$ to $\sigma=4$. Always $\Phi=0$.

\begin{figure}[tbp] 
\subfigure{\includegraphics[scale=0.41]{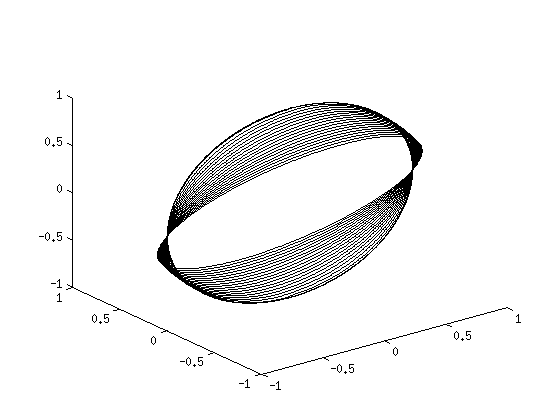}}
\subfigure{\includegraphics[scale=0.41]{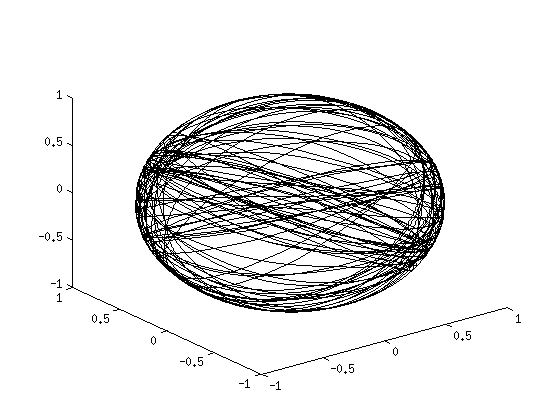}}
\subfigure{\includegraphics[scale=0.41]{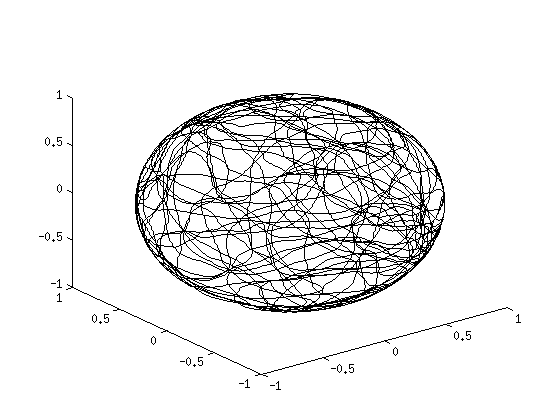}}
\subfigure{\includegraphics[scale=0.41]{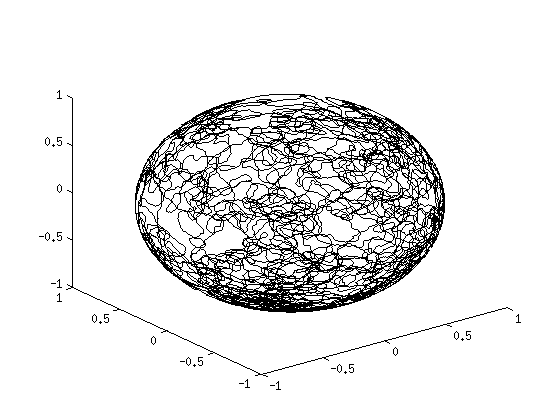}}
\caption{The geometric Langevin process with spherical velocity on $\mathbb{S}^2$} \label{figure_geometric_Langevin_Speed_sphere}
\end{figure}

\subsection{Appendix} We shall give the proof of Lemma \ref{Lm_generator_geometricLangevin_sphere} since the spherical Langevin process with spherical velocity can be of interest in specific applications, see the motivation from the introduction of this section.

\begin{proof}[Proof of Lemma \ref{Lm_generator_geometricLangevin_sphere}]
As before $\widetilde{\mathcal{A}}$ denotes the pushforward of some vector field $\mathcal{A}$. We have
\begin{align*} 
\widetilde{\frac{\partial}{\partial \theta_1}} \equiv \begin{pmatrix} \sin\theta_2\, n_1 \\  r\varphi_1 \,\partial_{\theta_1} n_1 + r\cos\theta_2 \varphi_2 \,n_1 \end{pmatrix},~\widetilde{\frac{\partial}{\partial \theta_2}} \equiv \begin{pmatrix} n_2 \\ -  r\varphi_2 \,\tau \end{pmatrix},~\widetilde{\frac{\partial}{\partial \alpha}} \equiv \begin{pmatrix} 0\\ - r \varphi_2 \,n_1 + r \varphi_1  \,n_2 \end{pmatrix}
\end{align*}
Let $x \in \mathbb{R}^3$ be arbitrary. Now there exists uniquely determined $a_1,\ldots,a_3 \in \mathbb{R}$ with
\begin{align*}
\begin{pmatrix} 0 \\ \Pi_{\mathbb{S}_r}[\omega] \Pi_{\mathbb{M}}[\xi] x \end{pmatrix} \equiv a_1 \widetilde{\frac{\partial}{\partial \theta_1}} + a_2 \widetilde{\frac{\partial}{\partial \theta_2}} + a_3\widetilde{\frac{\partial}{\partial \alpha}}.
\end{align*}
Immediately, we conclude that $a_1$ and $a_2$ must be zero. By taking the scalar product with respect to $\frac{1}{r^2}\widetilde{\partial_\alpha}$ on both sides of the last equation we get
\begin{align*}
a_3&= \frac{1}{r}\left( \Pi_{\mathbb{S}_r}[\omega] \Pi_{\mathbb{M}}[\xi] x ,   -  \varphi_2 n_1 + \varphi_1  n_2 \right)_{\text{euc}} \\
&= \frac{1}{r}\left( \Pi_{\mathbb{M}}[\xi] x ,   -  \varphi_2 n_1 + \varphi_1  n_2 \right)_{\text{euc}} = \frac{1}{r} \left( \varphi^\bot , \begin{pmatrix} x \cdot n_1 \\ x \cdot n_2 \end{pmatrix} \right)_{\text{euc}} .
\end{align*}
This follows since $\left( \omega, -  \varphi_2 n_1 + \varphi_1  n_2 \right)_{\text{euc}}=0$ and $\Pi_{\mathbb{M}}[\xi] x = x \cdot n_1 ~ n_1 + x \cdot n_2~ n_2$. Here $\varphi^\bot= \left( -\varphi_2, \varphi_1 \right)^T$. Let $\mathcal{A}_i = \begin{pmatrix} 0 \\ \Pi_{\mathbb{S}}[\omega] \Pi_{\mathbb{M}}[\xi] e_i \end{pmatrix}$, $e_i$ the $i$-th unit vector in $\mathbb{R}^3$, $i=1,\ldots,3$. By the previous calculation $\mathcal{A}_i$ writes on $\mathbb{Y} \times (0,\pi) \times \mathbb{Y}$ as
\begin{align*}
\mathcal{A}_i\equiv\frac{1}{r} \Big(\varphi^\bot, z_i \Big)_{\text{euc}} \, \frac{\partial}{\partial \alpha},~z_i:=\begin{pmatrix} e_i \cdot n_1 \\ e_i \cdot n_2 \end{pmatrix},~i=1,\ldots,3.
\end{align*}
Consequently, for each such $i$ we have
\begin{align*}
\mathcal{A}_i^2 =  - \frac{1}{r^2} \Big(\varphi, z_i \Big)_{\text{euc}} \Big(\varphi^\bot, z_i \Big)_{\text{euc}}\, \frac{\partial}{\partial \alpha} + \frac{1}{r^2} \Big(\varphi^\bot, z_i \Big)_{\text{euc}}^2\, \frac{\partial^2}{\partial \alpha^2}.
\end{align*}
An easy  calculation shows $\sum_{i} \Big(\varphi^\bot, z_i \Big)_{\text{euc}}^2 = 1$. Hence $\sum_{i} \frac{\partial  }{\partial \alpha} \Big(\varphi^\bot, z_i \Big)_{\text{euc}}^2 =0$. The latter means $\sum_{i} \Big(\varphi, z_i \Big)_{\text{euc}} \Big(\varphi^\bot, z_i \Big)_{\text{euc}} =0$. So altogether $\sum_{i} \mathcal{A}_i^2 =  \frac{1}{r^2} \frac{\partial^2}{\partial \alpha^2}$. Next observe that
\begin{align*}
\frac{\partial}{\partial {\theta_1}} \Psi = \sin \theta_2 \, n_1 \cdot \nabla_\xi \Phi,~\frac{\partial}{\partial {\theta_2}} \Psi = n_2 \cdot \nabla_\xi \Phi.
\end{align*}
So the vector field $\mathcal{V}_1:=-\begin{pmatrix} 0 \\ \Pi_{\mathbb{S}}[\omega] \Pi_{\mathbb{M}}[\xi] \nabla_\xi \Phi \end{pmatrix}$ writes on $\mathbb{Y} \times (0,\pi) \times \mathbb{Y}$  in the form
\begin{align*}
\mathcal{V}_1 \equiv - \frac{1}{r} \, \varphi^\bot \cdot \begin{pmatrix} \frac{1}{\sin \theta_2} \partial_{\theta_1} \Psi \\ \partial_{\theta_2} \Psi \end{pmatrix}  \,\frac{\partial}{\partial {\alpha}}.
\end{align*}
For the remaining vector field we again make the ansatz
\begin{align*}
\mathcal{V}_2:=\begin{pmatrix} \omega \\ - F(\xi, \omega) \end{pmatrix} \equiv b_1 \widetilde{\frac{\partial}{\partial \theta_1}} + b_2 \widetilde{\frac{\partial}{\partial \theta_2}} + b_3\widetilde{\frac{\partial}{\partial \alpha}}
\end{align*}
for some uniquely determined $b_i \in \mathbb{R}$, $i=1,\ldots,3$. Here recall the definition of $F$ in Section \ref{Langevin_process_constant_velocity_version}. Thus directly $b_1=  r \frac{\varphi_1}{\sin \theta_2}$ and $b_2= r \varphi_2$. Furthermore, one easily checks that $F(\xi,\omega)=r^2\xi=r^2\,\tau$. By taking again the scalar product with respect to $\frac{1}{r^2}\widetilde{\partial_\alpha}$ on both sides of the last equation we obtain
\begin{align*}
0 = b_1 \left( \varphi_1 \,\partial_{\theta_1} n_1 + \cos\theta_2 \varphi_2 \,n_1,-\varphi_2 n_1 + \varphi_1 n_2\right)_{\text{euc}} + b_3 = - b_1 \cos \theta_2 \left( \varphi_2^2 + \varphi_1^2 \right) + b_3.
\end{align*}
And therefore $b_3 = b_1 \cos \theta_2 =r \cos \alpha \cot \theta_2$. Altogether, the claim follows since the the generator $L$ is equal to $L=\mathcal{V}_1 + \mathcal{V}_2 + \frac{\sigma^2}{2} \sum_{i=1}^3 \mathcal{A}_i^2$.
\end{proof}

\section*{Acknowledgement} This work has been supported by Bundesministerium f"{u}r Bildung und Forschung, Schwerpunkt \glqq Mathematik f"{u}r Innovationen in Industrie and Dienstleistungen\grqq , Verbundprojekt  ProFil, $03$MS$606$. Furthermore, we thank Axel Klar and Johannes Maringer for stimulating discussions as well as Wolfgang Bock, Uditha Prabhath Liyanage and Johannes Maringer for helping us with some numerical simulations.

\bibliographystyle{alpha}

\end{document}